\documentclass[11pt,letterpaper]{amsart}
\usepackage[margin=1in]{geometry}

\usepackage{amsmath,amssymb,amsthm,mathrsfs,bbm,comment,units,enumitem,xcolor}
\usepackage[colorlinks, linktocpage]{hyperref}
\usepackage{pgf,tikz}
\usepackage{tikz-3dplot}
\usetikzlibrary{positioning}
\usetikzlibrary{decorations.pathreplacing}
\usetikzlibrary{arrows,shapes}
\usetikzlibrary{arrows.meta,bending,automata}

\renewcommand{\leq}{\leqslant}
\renewcommand{\geq}{\geqslant}
\newcommand{\concat}{\overset{\frown}{}}

\newcommand{\One}{\text{I}}
\newcommand{\Two}{\text{II}}

\newcommand{\Tone}{\overset{\leftarrow}{T}_{\One, n}}
\newcommand{\Ttwo}{\overset{\rightarrow}{T}_{\Two, n}}

\theoremstyle{definition}
\newtheorem{theorem}{Theorem}
\newtheorem{definition}{Definition}
\newtheorem{lemma}[theorem]{Lemma}
\newtheorem{corollary}[theorem]{Corollary}

\newtheorem{remark}{Remark}

\newtheorem*{example}{Example}

\title{Selection Games and the Vietoris Space}
\author{Christopher Caruvana}
\author{Jared Holshouser}
\keywords{Vietoris topology, Selection principles, Rothberger property, Menger property}
\subjclass{54D20, 54D80}
\date{\today}

\begin{document}

\maketitle

\begin{abstract}
    We explore the connections between selection games on Hausdorff spaces and their corresponding Vietoris space of compact subsets.
    These considerations offer a similar relationship as the well-known relationship between \(\omega\)-covers of \(X\) and ordinary open covers of the finite powers of \(X\).
    The primary utility of this method is to establish similar relationships with \(k\)-covers and the Vietoris space of compact subsets.
    Particularly, we show that some commonly studied selection principles are equivalent to a related hyperspace being Menger or Rothberger.
    We then apply these equivalences to correct a flawed argument in a previous paper which attempted to show that a Pawlikowski theorem is true for \(k\)-covers.
\end{abstract}

\section{Introduction}

Relationships between selection principles on a ground space and the hyperspace of closed subsets with various topologies has been a growing area of investigation \cite{CHHyperspaces,Caserta,KocinacEtAl2005,Li2016,Mrsevic,Osipov}.
One of the common techniques employed is to translate certain cover types to families of closed sets via the complement operation.
The resulting relationship is thus between covers of a certain type and dense sets in a related topology given to the space of closed sets.
In \cite[Theorem 3.22]{CHCompactOpen}, a more direct topological relationship is suggested.
By restricting our attention to compact sets, we bring to bear relationships between a space \(X\) and the space \(\mathbb K(X)\) of compact subsets endowed with the Vietoris topology in terms of the cover types themselves.
In particular, we establish relationships between \(\omega\)-covers and open covers on the space of finite subsets of \(X\) viewed as a subspace of \(\mathbb K(X)\) as well as relationships between \(k\)-covers and open covers of \(\mathbb K(X)\).

We also point out an application of these methods, following existing results of \cite{GerlitsNagy,JustMillerScheepers,Sakai1988}, to prove strategic equivalence of some Rothberger- and Menger-like games on \(X\) with the corresponding games on the disjoint union \(X^{<\omega}\) of finite powers of \(X\).
Particularly, these classical results establish a relationship between selection principles involving \(\omega\)-covers on \(X\) and open covers on \(X^{<\omega}\).
This is natural since \(\omega\)-covers are to cover all finite subsets of \(X\) and one can code finite subsets of \(X\) with tuples.

Finally, we also address the following.
Steven Clontz pointed out that the without loss of generality claim in the beginning of the proof of \cite[Proposition 3.25]{CHCompactOpen} (restated in a slightly generalized version in \cite{CHContinuousFunctions} as Lemma 7) is flawed.
We then noticed a similar flaw at the end of the proof of \cite[Proposition 3.27]{CHCompactOpen} (restated in a slightly generalized version in \cite{CHContinuousFunctions} as Lemma 8).
In this note, we recover the conclusions of those results for \(k\)-covers.
However, we remain with the following question.
Are \cite[Lemma 7]{CHContinuousFunctions} and \cite[Lemma 8]{CHContinuousFunctions} true as stated?

Throughout, we assume that all spaces \(X\) considered are Hausdorff, infinite, and, when relevant, non-compact.

\section{Preliminaries}

\begin{definition}
    For a topological space \(X\), we let \(\mathscr T_X\) denote the collection of all proper, non-empty open subsets of \(X\).
\end{definition}
\begin{definition}
    Generally, for an open cover \(\mathscr U\) of a topological space \(X\), we say that \(\mathscr U\) is \emph{non-trivial} provided that \(X \not\in \mathscr U\).
    We let \(\mathcal O_X\) denote the collection of all non-trivial open covers of \(X\).
\end{definition}
\begin{definition}
    For a space \(X\) and a class \(\mathcal A\) of closed proper subsets of \(X\), a non-trivial open cover \(\mathscr U\) is an \emph{\(\mathcal A\)-cover} if, for every \(A \in \mathcal A\), there exists \(U \in \mathscr U\) so that \(A \subseteq U\).
    We let \(\mathcal O(X, \mathcal A)\) denote the collection of all \(\mathcal A\)-covers of \(X\).
\end{definition}
\begin{remark}
    Note that
    \begin{itemize}
        \item
		if \(\mathcal A\) consists of the finite subsets of \(X\), then \(\mathcal O(X,\mathcal A)\) is the collection of all \(\omega\)-covers of \(X\), which will be denoted by \(\Omega_X\).
		\item
		if \(\mathcal A\) consists of the compact (proper) subsets of \(X\), then \(\mathcal O(X, \mathcal A)\) is the collection of all \(k\)-covers of \(X\), which will be denoted by \(\mathcal K_X\).
    \end{itemize}
\end{remark}
\begin{definition}
	Given a set \(\mathcal A\) and another set \(\mathcal B\), we define the \emph{finite selection game} \(\textsf{G}_{\text{fin}}(\mathcal A, \mathcal B)\) for \(\mathcal A\) and \(\mathcal B\) as follows:
	\[
		\begin{array}{c|cccc}
			\text{I} & A_0 & A_1 & A_2 & \ldots \\
			\hline
			\text{II} & \mathcal F_0 & \mathcal F_1 & \mathcal F_2 & \ldots
		\end{array}
	\]
	where \(A_n \in \mathcal A\) and \(\mathcal F_n \in [A_n]^{<\omega}\) for all \(n \in \omega\).
	We declare Two the winner if \(\bigcup\{ \mathcal F_n : n \in \omega \} \in \mathcal B\).
	Otherwise, One wins.
\end{definition}

\begin{definition}
	Similarly, we define the \emph{single selection game} \(\textsf{G}_1(\mathcal A, \mathcal B)\) as follows:
	\[
		\begin{array}{c|cccc}
			\text{I} & A_0 & A_1 & A_2 & \ldots \\
			\hline
			\text{II} & x_0 & x_1 & x_2 & \ldots
		\end{array}
	\]
	where each \(A_n \in \mathcal A\) and \(x_n \in A_n\).
	We declare Two the winner if \(\{ x_n : n \in \omega \} \in \mathcal B\).
	Otherwise, One wins.
\end{definition}

\begin{definition}
    We define strategies of various strength below.
    \begin{itemize}
    \item
    A \emph{strategy for player One} in \(\textsf{G}_1(\mathcal A, \mathcal B)\) is a function \(\sigma:(\bigcup \mathcal A)^{<\omega} \to \mathcal A\).
    A strategy \(\sigma\) for One is called \emph{winning} if whenever \(x_n \in \sigma\langle x_k : k < n \rangle\) for all \(n \in \omega\), \(\{x_n: n\in\omega\} \not\in \mathcal B\).
    If player One has a winning strategy, we write \(\One \uparrow \textsf{G}_1(\mathcal A, \mathcal B)\).
    \item
    A \emph{strategy for player Two} in \(\textsf{G}_1(\mathcal A, \mathcal B)\) is a function \(\tau:\mathcal A^{<\omega} \to \bigcup \mathcal A\).
    A strategy \(\tau\) for Two is \emph{winning} if whenever \(A_n \in \mathcal A\) for all \(n \in \omega\), \(\{\tau(A_0,\ldots,A_n) : n \in \omega\} \in \mathcal B\).
    If player Two has a winning strategy, we write \(\Two \uparrow \textsf{G}_1(\mathcal A, \mathcal B)\).
    \item
    A \emph{predetermined strategy} for One is a strategy which only considers the current turn number.
    We call this kind of strategy predetermined because One is not reacting to Two's moves, they are just running through a pre-planned script.
    Formally it is a function \(\sigma: \omega \to \mathcal A\).
    If One has a winning predetermined strategy, we write \(\One \underset{\text{pre}}{\uparrow} \textsf{G}_1(\mathcal A, \mathcal B)\).
    \item
    A \emph{Markov strategy} for Two is a strategy which only considers the most recent move of player One and the current turn number.
    Formally it is a function \(\tau:\mathcal A \times \omega \to \bigcup \mathcal A\).
    If Two has a winning Markov strategy, we write \(\Two \underset{\text{mark}}{\uparrow} \textsf{G}_1(\mathcal A, \mathcal B)\).
    \item
    If there is a single element \(x_0 \in \mathcal A\) so that the constant function with value \(x_0\) is a winning strategy for One, we say that One has a \emph{constant winning strategy}, denoted by \(\One \underset{\text{cnst}}{\uparrow} \textsf{G}_1(\mathcal A, \mathcal B)\).
    \end{itemize}
    These definitions can be extended to \(\textsf{G}_{\text{fin}}(\mathcal A, \mathcal B)\) in the obvious way.
\end{definition}

\begin{definition}
    The reader may be more familiar with selection principles than selection games.
    Let \(\mathcal A\) and \(\mathcal B\) be collections.
    The selection principle \(\textsf{S}_1(\mathcal A, \mathcal B)\) for a space \(X\) is the following property:
    Given any sequence \(\langle A_n : n \in \omega \rangle\) from \(\mathcal A\), there exists \(\{x_n : n \in \omega \}\) with \(x_n \in A_n\) for each \(n \in \omega\) so that \(\{ x_n : n \in \omega \} \in \mathcal B\).
    \(\textsf{S}_{\text{fin}}(\mathcal A, \mathcal B)\) is similarly defined, but with finite selections instead of single selections.
    We will use the notation \(X \models \textsf{S}_\square(\mathcal A, \mathcal B)\) to denote that the selection principle \(\textsf{S}_\square(\mathcal A, \mathcal B)\) holds for \(X\).
\end{definition}
Note that
\begin{itemize}
    \item \(\textsf{S}_{\text{fin}}(\mathcal O, \mathcal O)\) is the Menger property.
    \item \(\textsf{S}_{1}(\mathcal O, \mathcal O)\) is the Rothberger property.
\end{itemize}

\begin{remark}
    In general, \(\textsf{S}_\square(\mathcal A, \mathcal B)\) holds if and only if \(\text{I} \underset{\text{pre}}{\not\uparrow} \textsf{G}_\square(\mathcal{A},\mathcal{B})\) where \(\square \in \{1 , \text{fin} \}\).
    See \cite[Prop. 15]{ClontzDuality}.
\end{remark}

\begin{definition}
    An even more fundamental type of selection is inspired by the Lindel\"{o}f property.
    Let \(\mathcal A\) and \(\mathcal B\) be collections.
    Then \(\genfrac(){0pt}{1}{\mathcal A}{\mathcal B}\) means that, for every \(A \in \mathcal A\), there exists \(B \subseteq A\) so that \(B \in \mathcal B\).
    Scheepers calls this a \emph{Bar-Ilan selection principle} in \cite{Scheepers2003}.
\end{definition}

\begin{remark} \label{rmk:ClontzBarIlan}
    Let \(\mathcal A\) and \(\mathcal B\) be collections.
    We let \(\text{ctbl}(\mathcal B) = \{B \in \mathcal B: |B| \leq \omega\}\).
    Then One fails to have a constant strategy in \(\textsf{G}_1(\mathcal A, \mathcal B)\) if and only if \(\genfrac(){0pt}{1}{\mathcal A}{\text{ctbl}(\mathcal B)}\) holds as shown in \cite[Prop. 15]{ClontzDuality}.
\end{remark}
In fact,
\begin{lemma} \label{lem:Lindelof}
    For any space \(X\),
    \[
        X \models \genfrac(){0pt}{0}{\mathcal A}{\text{ctbl}(\mathcal B)}
        \iff \text{I} \underset{\text{cnst}}{\not\uparrow} \textsf{G}_1(\mathcal A, \mathcal B)
        \iff \text{I} \underset{\text{cnst}}{\not\uparrow} \textsf{G}_{\text{fin}}(\mathcal A, \mathcal B)
    \]
\end{lemma}
\begin{proof}
    By Remark \ref{rmk:ClontzBarIlan}, the only thing to show is the equivalence of the non-existence of a constant strategy for One in the single selection and finite selection games.
    This equivalence can be seen by the fact that any play by Two in the finite selection game can be translated to a play in the single selection game since a countable collection of finite sets is countable.
\end{proof}

Note that, in the language of \cite{KocinacSelectedResults},
\begin{itemize}
    \item
    \(\genfrac(){0pt}{0}{\mathcal O}{\text{ctbl}(\mathcal O)}\) is the Lindel{\"{o}}f property.
    \item
    \(\genfrac(){0pt}{0}{\Omega}{\text{ctbl}(\Omega)}\) is the \(\omega\)-Lindel{\"{o}}f property, most commonly known as the \(\epsilon\)-space property.
    \item
    \(\genfrac(){0pt}{0}{\mathcal K}{\text{ctbl}(\mathcal K)}\) is the \(k\)-Lindel{\"{o}}f property.
\end{itemize}

\begin{definition}
We say that \(\mathcal G\) is a \emph{selection game} if there exist classes \(\mathcal A\), \(\mathcal B\) and \(\square \in \{ 1 , \text{fin}\}\) so that \(\mathcal G = \textsf{G}_\square(\mathcal A, \mathcal B)\).
\end{definition}

\subsection{Game-theoretic Tools}

\begin{definition}
  We say that two selection games \(\mathcal G\) and \(\mathcal H\) are \emph{equivalent}, denoted \(\mathcal G \equiv \mathcal H\), if the following hold:
  \begin{itemize}
      \item
      \(\text{II} \underset{\text{mark}}{\uparrow} \mathcal G \iff \text{II} \underset{\text{mark}}{\uparrow} \mathcal H\)
      \item
      \(\text{II} \uparrow \mathcal G \iff \text{II} \uparrow \mathcal H\)
      \item
      \(\text{I} \not\uparrow \mathcal G \iff \text{I} \not\uparrow \mathcal H\)
      \item
      \(\text{I} \underset{\text{pre}}{\not\uparrow} \mathcal G \iff \text{I} \underset{\text{pre}}{\not\uparrow} \mathcal H\)
      \item
      \(\text{I} \underset{\text{cnst}}{\not\uparrow} \mathcal G \iff \text{I} \underset{\text{cnst}}{\not\uparrow} \mathcal H\)
  \end{itemize}
\end{definition}

\begin{definition}
  Given selection games \(\mathcal G\) and \(\mathcal H\), we say that \(\mathcal G \leq_{\text{II}} \mathcal H\) if the following implications hold:
  \begin{itemize}
      \item
      \(\text{II} \underset{\text{mark}}{\uparrow} \mathcal G \implies \text{II} \underset{\text{mark}}{\uparrow} \mathcal H\)
      \item
      \(\text{II} \uparrow \mathcal G \implies \text{II} \uparrow \mathcal H\)
      \item
      \(\text{I} \not\uparrow \mathcal G \implies \text{I} \not\uparrow \mathcal H\)
      \item
      \(\text{I} \underset{\text{pre}}{\not\uparrow} \mathcal G \implies \text{I} \underset{\text{pre}}{\not\uparrow} \mathcal H\)
      \item
        \(\text{I} \underset{\text{cnst}}{\not\uparrow} \mathcal G \implies \text{I} \underset{\text{cnst}}{\not\uparrow} \mathcal H\)
  \end{itemize}
\end{definition}
Note that \(\leq_{\text{II}}\) is transitive and that if \(\mathcal G \leq_{\text{II}} \mathcal H\) and \(\mathcal H \leq_{\text{II}} \mathcal G\), then \(\mathcal G \equiv \mathcal H\).
We use the subscript of II since each implication is related to a transference of winning plays by Two.
Also, for classes \(\mathcal A\) and \(\mathcal B\),
\[
    \textsf{G}_1(\mathcal A, \mathcal B) \leq_{\text{II}} \textsf{G}_{\text{fin}}(\mathcal A , \mathcal B).
\]

We now recall the Translation Theorems that will be relevant in the sequel.

\begin{theorem}[\cite{CHHyperspaces}] \label{Theorem16}
    Let \(\mathcal A\), \(\mathcal B\), \(\mathcal C\), and \(\mathcal D\) be collections.
    Suppose there are functions
    \begin{itemize}
        \item \(\Tone:\mathcal B \to \mathcal A\) and
        \item \(\Ttwo:\left[\bigcup \mathcal A \right]^{<\omega} \times \mathcal B \to \left[\bigcup \mathcal B \right]^{<\omega}\)
    \end{itemize}
    for each \(n \in \omega\) so that
    \begin{enumerate}[label=(P\arabic*)]
        \item \label{FiniteTransA} If \(\mathcal F \in [\Tone(B)]^{<\omega}\), then \(\Ttwo(\mathcal F,B) \in [B]^{<\omega}\)
        \item \label{FiniteTransB} If \(\mathcal F_n \in [\Tone(B_n)]^{<\omega}\) for each \(n \in \omega\) and \(\bigcup_{n \in \omega} \mathcal F_n \in \mathcal C\), then \(\bigcup_{n \in \omega} \Ttwo(\mathcal F_n,B_n) \in \mathcal D\).
    \end{enumerate}
    Then \(\textsf{G}_{\text{fin}}(\mathcal A, \mathcal C) \leq_{\Two} \textsf{G}_{\text{fin}}(\mathcal B, \mathcal D)\).
\end{theorem}
\begin{proof}
    Most of of the proof of this is in \cite[Theorem 16]{CHHyperspaces}.
    The only thing that remains to be proved is the implication
    \[
        \text{I} \underset{\text{cnst}}{\not\uparrow} \textsf{G}_{\text{fin}}(\mathcal A, \mathcal C) \implies \text{I} \underset{\text{cnst}}{\not\uparrow} \textsf{G}_{\text{fin}}(\mathcal B, \mathcal D).
    \]
    Suppose One does not have a constant winning strategy in \(\textsf{G}_{\text{fin}}(\mathcal A, \mathcal C)\) and let \(B \in \mathcal B\) be arbitrary.
    As \(\Tone(B) \in \mathcal A\), there exist \(\mathcal F_n \in \left[ \Tone(B) \right]^{<\omega}\) so that \(\bigcup_{n\in\omega} \mathcal F_n \in \mathcal C\).
    Hence, \(\bigcup_{n\in\omega} \Ttwo(\mathcal F_n , B) \in \mathcal D\).
    As \(B \in \mathcal B\) was arbitrary, we see that One does not have a constant winning strategy in \(\textsf{G}_{\text{fin}}(\mathcal B, \mathcal D)\).
\end{proof}

\begin{corollary}[\cite{CHHyperspaces}] \label{Corollary17}
    Let \(\mathcal A\), \(\mathcal B\), \(\mathcal C\), and \(\mathcal D\) be collections.
    Suppose there are functions
    \begin{itemize}
        \item \(\Tone :\mathcal B \to \mathcal A\) and
        \item \(\Ttwo : \left(\bigcup \mathcal A \right) \times \mathcal B \to \bigcup \mathcal B\)
    \end{itemize}
    for each \(n \in \omega\) so that the following two properties hold.
    \begin{enumerate}[label={(Ft\arabic*)}]
        \item \label{translationPropI} If \(x \in \Tone(B)\), then \(\Ttwo(x,B) \in B\).
        \item \label{translationPropII} If \(\mathcal F_n \in \left[\Tone(B_n)\right]^{<\omega}\) and \(\bigcup_{n \in \omega} \mathcal F_n \in \mathcal C\), then \(\bigcup_{n \in \omega} \left\{ \Ttwo(x,B_n) : x \in \mathcal F_n \right\} \in \mathcal D\).
    \end{enumerate}
    Then, for \(\square \in \{1,\text{fin}\}\), \(G_\square(\mathcal A,\mathcal C) \leq_{\Two} G_\square(\mathcal B, \mathcal D)\).
\end{corollary}

\begin{corollary}[\cite{CHHyperspaces}] \label{Corollary18}
    Let \(\mathcal A\), \(\mathcal B\), \(\mathcal C\), and \(\mathcal D\) be collections.
    Suppose there is a map \(\varphi:\left[\bigcup \mathcal B\right] \times \omega \to \bigcup \mathcal A\) so that the following two conditions hold.
    \begin{itemize}
        \item For all \(B \in \mathcal B\) and all \(n\in\omega\), \(\{\varphi(y,n) : y \in B\} \in \mathcal A\).
        \item If \(\mathcal G_n \in [B_n]^{<\omega}\) where \(B_n \in \mathcal B\) for each \(n \in \omega\) and \(\bigcup_{n \in \omega} \varphi[\mathcal G_n \times \{n\}] \in \mathcal C\), then \(\bigcup_{n \in \omega} \mathcal G_n \in \mathcal D\).
    \end{itemize}
    Then, for \(\square \in \{1,\text{fin}\}\), \(G_\square(\mathcal A,\mathcal C) \leq_{\Two} G_\square(\mathcal B, \mathcal D)\).
\end{corollary}

\begin{definition}
    Consider a class \(\mathcal C\) and a collection \(C \in \mathcal C\). We say that \(C'\) is an \emph{enlargement} of \(C\) if \(C' \subseteq \bigcup \mathcal C\) and \((\forall x \in C)(\exists y \in C^\prime)[x \subseteq y]\).
\end{definition}
\begin{definition}
    We say that a class \(\mathcal C\) is \emph{closed under enlargement} if the following property holds: if \(C \in \mathcal C\) and \(C^\prime\) is an enlargement of \(C\), then \(C^\prime \in \mathcal C\).
\end{definition}
Note that \(\mathcal O(X,\mathcal A)\) for any family of closed sets \(\mathcal A\) of a space \(X\) is closed under enlargement.
\begin{definition}
    Let \(\mathcal A\) and \(\mathcal B\) be classes and \(\varphi : \bigcup \mathcal B \to \bigcup \mathcal A\).
    For \(A \in \mathcal A\), we define the \emph{\(\varphi\)-refinement of A} to be
    \[
        \left\{ y \in \bigcup \mathcal B :(\exists x \in A)[\varphi(y) \subseteq x] \right\}.
    \]
\end{definition}
\begin{corollary} \label{bestCorollary}
    Suppose \(\mathcal A\), \(\mathcal B\), \(\mathcal C\), and \(\mathcal D\) are classes so that \(\bigcup \mathcal C \subseteq \bigcup \mathcal A\) and \(\bigcup \mathcal D \subseteq \bigcup \mathcal B\).
    Suppose there is a function \(\varphi : \bigcup \mathcal B \to \bigcup \mathcal A\) so that, for any \(B \in \mathcal B\), \(\varphi[B] \in \mathcal A\) and, for any \(E \subseteq \bigcup \mathcal B\) so that \(\varphi[E] \in \mathcal C\), \(E \in \mathcal D\).
    Then, for \(\square \in \{ 1 , \text{fin} \}\), \(\textsf{G}_\square(\mathcal A, \mathcal C) \leq_{\text{II}} \textsf{G}_\square(\mathcal B, \mathcal D)\).

    If, in addition,
    \begin{itemize}
        \item
        \(\mathcal C\) is closed under enlargement,
        \item
        for any \(A \in \mathcal A\), the \(\varphi\)-refinement of \(A\) is an element of \(\mathcal B\), and
        \item
        for any \(E \subseteq \bigcup \mathcal B\), \(E \in \mathcal D\) implies \(\varphi[E] \in \mathcal C\),
    \end{itemize}
    then \(\textsf{G}_\square(\mathcal A, \mathcal C) \equiv \textsf{G}_\square(\mathcal B, \mathcal D)\).
\end{corollary}
\begin{proof}
    The first condition of Corollary \ref{Corollary18} holds.
    Now, suppose \(\mathcal G_n \in [B_n]^{<\omega}\) where \(B_n \in \mathcal B\) for all \(n \in \omega\) is so that \(\bigcup_{n\in\omega} \varphi[\mathcal G_n] \in \mathcal C\).
    Then \(\bigcup_{n\in\omega} \varphi[\mathcal G_n] = \varphi\left[\bigcup_{n\in\omega} \mathcal G \right]\in \mathcal C\) implies that \(\bigcup_{n\in\omega}\mathcal G_n \in\mathcal D\).
    Hence, \(\textsf{G}_\square(\mathcal A, \mathcal C) \leq_{\text{II}} \textsf{G}_\square(\mathcal B, \mathcal D)\).

    For the remainder, we use Corollary \ref{Corollary17}.
    Define \(\Tone: \mathcal A \to \mathcal B\) to be the \(\varphi\)-refinment of \(\mathcal A\).
    Then define \(\Ttwo : (\bigcup \mathcal B) \times \mathcal A \to \bigcup \mathcal A\) in the following way.
    If \(y \in \bigcup \mathcal B\) and \(A \in \mathcal A\) are so that there exists \(x \in A\) with \(\varphi(y) \subseteq x\), let \(\Ttwo(y,A) \in A\) be so that \(\varphi(y) \subseteq \Ttwo(y,A)\).
    Otherwise, let \(\Ttwo(y,A) = y\).

    By our definitions, if \(y \in \Tone(A)\), then \(\Ttwo(y,A) \in A\).
    So suppose, for every \(n \in \omega\), \(y_{1,n} , \ldots , y_{k_n,n} \in \Tone(A_n)\) are so that \(\bigcup_{n\in\omega} \{y_{1,n} , \ldots , y_{k_n,n}\} \in \mathcal D\).
    By the hypotheses, we have that \(\bigcup_{n\in\omega} \{\varphi(y_{1,n}) , \ldots , \varphi(y_{k_n,n})\} \in \mathcal C\).
    Observe that \(\varphi(y_{j,n}) \subseteq \Ttwo(y_{j,n} , A_n)\) for any \(n\in\omega\) and \(1 \leq j \leq k_n\) which provides that \(\bigcup_{n\in\omega}\{\Ttwo(y_{1,n},A_n) , \ldots , \Ttwo(y_{k_n,n},A_n)\} \in \mathcal C\) since \(\mathcal C\) is closed under enlargement.
\end{proof}

As we will see in the applications, Corollary \ref{bestCorollary} is capturing game equivalence under the condition that there is an adequate way to translate between cover types via some translation of open sets.

As an introductory application, the translation of winning plays for Two is monotone with respect to closed subspaces, just as one would expect.
\begin{lemma} \label{lem:SubspaceMonotone}
    Let \(X\) be a space, \(\mathcal A\) be a family of closed subsets of \(X\), and \(Y \subseteq X\) be closed so that \(Y \not\in \mathcal A\).
    Then \(\mathcal B := \{ A \cap Y : A \in \mathcal A \}\) is a family of closed subsets of \(Y\) and, for \(\square \in \{1,\text{fin}\}\),
    \[
        \textsf{G}_\square(\mathcal O(X,\mathcal A), \mathcal O(X,\mathcal A))
        \leq_{\text{II}} \textsf{G}_\square(\mathcal O(Y,\mathcal B), \mathcal O(Y,\mathcal B)).
    \]
\end{lemma}
\begin{proof}
    First, for any open \(V \subseteq Y\), let \(W_V\) be open in \(X\) so that \(W_V \cap Y = V\).
    Then define \(\varphi : \mathscr T_Y \to \mathscr T_X\) by the rule \(\varphi(V) = W_V \cup (X \setminus Y)\).
    If \(\mathscr V \in \mathcal O(Y,\mathcal B)\), we show that \(\varphi[\mathscr V] \in \mathcal O(X,\mathcal A)\).
    Let \(A \in \mathcal A\) and find \(V \in \mathscr V\) so that \(A \cap Y \subseteq V\).
    Then \(A \subseteq \varphi(V)\).

    Now, suppose \(\mathscr E\) is a collection of open subsets of \(Y\) so that \(\varphi[\mathscr E] \in \mathcal O(X,\mathcal A)\).
    We show that \(\mathscr E \in \mathcal O(Y,\mathcal B)\).
    Let \(B \in \mathcal B\) and \(A \in \mathcal A\) be so that \(B = A \cap Y\).
    There is some \(E \in \mathscr E\) so that \(A \subseteq \varphi(E)\) and so we see that \(B = A \cap Y \subseteq \varphi(E) \cap Y = E\).
    So Corollary \ref{bestCorollary} applies.
\end{proof}

Notice that Two has a winning Markov strategy in \(\textsf{G}_1(\mathcal K_{\mathbb R} , \mathcal K_{\mathbb R})\) and since \(\mathbb Q\) is not hemicompact, by \cite[Theorem 3.22]{CHCompactOpen}, Two does not have a winning Markov strategy in \(\textsf{G}_1(\mathcal K_{\mathbb Q} , \mathcal K_{\mathbb Q})\).
Thus, \(\textsf{G}_1(\mathcal K_{\mathbb R} , \mathcal K_{\mathbb R}) \not\leq_{\text{II}} \textsf{G}_1(\mathcal K_{\mathbb Q} , \mathcal K_{\mathbb Q})\) so the requirement that the subspace be closed is necessary.

Similarly, the inequality does not reverse as Two has a winning Markov strategy in \(\textsf{G}_1(\mathcal O_{\mathbb Z} , \mathcal O_{\mathbb Z})\) but does not have a Markov winning strategy in \(\textsf{G}_1(\mathcal O_{\mathbb R} , \mathcal O_{\mathbb R})\), proving \(\textsf{G}_1(\mathcal O_{\mathbb Z} , \mathcal O_{\mathbb Z}) \not\leq_{\text{II}} \textsf{G}_1(\mathcal O_{\mathbb R} , \mathcal O_{\mathbb R})\).

\section{Applications to Finite Powers} \label{section:Inspiration}

Let \(X^{<\omega}\) be the disjoint union of all \(X^n\) for \(n \geq 1\).
Clearly, \(X^{<\omega}\) is a coding space for all finite subsets of \(X\) so one may anticipate a relationship between open covers of \(X^{<\omega}\) and \(\omega\)-covers of \(X\).
Indeed, we revisit those well-known connections.

The following result concerning \(\omega\)-covers and how they interact with finite powers can be seen as the real driving force behind the results of this section and, moreover, the inspiration behind Lemma \ref{lem:CoverTranslation}.
\begin{lemma}[Adapted from Lemmas 3.2 and 3.3 of \cite{JustMillerScheepers}] \label{lem:RectangleRefinement}
    Let \(X\) be a space and \(n \geq 1\).
    Then,
    \begin{enumerate}[label=(\alph*)]
        \item \label{RC}
        if \(\mathscr U\) is an \(\omega\)-cover of \(X\), then \(\{ U^n : U \in \mathscr U\}\) is an \(\omega\)-cover of \(X^n\).
        \item \label{RR}
        if \(\mathscr U\) is an \(\omega\)-cover of \(X^n\),
        \[
            \mathscr V = \{ V \in \mathscr T_X : (\exists U \in \mathscr U)[V^n \subseteq U] \}
        \]
        is an \(\omega\)-cover of \(X\).
    \end{enumerate}
\end{lemma}
\begin{lemma} \label{lem:TubeRefinement}
    Let \(X\) be a space and \(n \geq 1\).
    For any \(A \subseteq X\), define \(A^{\leq n}\) be the disjoint union of \(A, A^2 , \ldots, A^n\).
    Then,
    \begin{enumerate}[label=(\alph*)]
        \item \label{TRC}
        if \(\mathscr U\) is an \(\omega\)-cover of \(X\), then \(\{ U^{\leq n} : U \in \mathscr U\}\) is an \(\omega\)-cover of \(X^{\leq n}\).
        \item \label{TRR}
        if \(\mathscr U\) is an \(\omega\)-cover of \(X^{\leq n}\),
        \[
            \mathscr V = \left\{ V \in \mathscr T_X : (\exists U \in \mathscr U)\left[V^{\leq n} \subseteq U\right] \right\}
        \]
        is an \(\omega\)-cover of \(X\).
    \end{enumerate}
\end{lemma}
\begin{proof}
    Though the proof here is similar to a proof of Lemma \ref{lem:RectangleRefinement}, we provide it in full for the convenience of the reader.

    \ref{TRC}
    Suppose \(\mathscr U\) is an \(\omega\)-cover of \(X\) and let \(\mathscr F\) be any finite subset of \(X^{\leq n}\).
    Notice that \[p(\mathscr F) := \{ x \in X : (\exists \mathbf x \in \mathscr F)(\exists j \in \omega)[\pi_j(\mathbf x) = x] \}\] is a finite subset of \(X\) where \(\pi_j\) is the usual projection onto the \(j^{\text{th}}\) coordinate.
    Then we can find \(U \in \mathscr U\) so that \(p(\mathscr F) \subseteq U\).
    Notice that \(\mathscr F \subseteq U^{\leq n}\).

    \ref{TRR}
    Now suppose \(\mathscr U\) is an \(\omega\)-cover of \(X^{\leq n}\) and let \(F = \{ x_1 , x_2 , \ldots , x_m \} \subseteq X\).
    Certainly, \(F^{\leq n}\) is a finite subset of \(X^{\leq n}\) so there exists \(U \in \mathscr U\) so that \(F^{\leq n} \subseteq U\).
    For any \(\vec y = (y_1 , y_2 , \ldots , y_k) \in F^{\leq n}\), let \(V_1(\vec y) , \ldots , V_k(\vec y)\) be so that
    \[
        (y_1 , y_2 , \ldots , y_k ) \in \prod_{j=1}^k V_j(\vec y) \subseteq U.
    \]
    Observe that \[\mathscr V = \left\{ V_j(\vec y) : \vec y \in F^{\leq n}, 1 \leq j \leq \text{len}(\vec y) \right\}\] is a finite collection of open subsets.
    So, for \(1 \leq \ell \leq m\), define \(W_\ell = \bigcap \{ V \in \mathscr V : x_\ell \in V \}\) and then \(W = \bigcup_{\ell=1}^m W_\ell\).
    Clearly, \(W\) is an open subset of \(X\) and \(F \subseteq W\).

    The only thing that remains to be shown is that \(W^{\leq n} \subseteq U\).
    For \(1 \leq k \leq n\), consider \((y_1,\ldots, y_k) \in W^k\).
    Let \(1 \leq \ell_1 , \ell_2, \ldots , \ell_k \leq m\) be so that \(y_j \in W_{\ell_j}\) for each \(1 \leq j \leq k\).
    We can now note that
    \[
        \vec x = \left( x_{\ell_1} , x_{\ell_2} , \ldots , x_{\ell_k} \right) \in \prod_{j=1}^k V_j(\vec x) \subseteq U.
    \]
    As \(W_{\ell_j} \subseteq V_j(\vec x)\), we see that
    \[
        (y_1,\ldots, y_k) \in \prod_{j=1}^k W_{\ell_j} \subseteq \prod_{j=1}^k V_j(\vec x) \subseteq U.
    \]
    As \(k\) was chosen to be arbitrary, the proof is finished.
\end{proof}

\begin{lemma} \label{lem:FirstFinitePower}
    For any space \(X\), \(n \geq 1\), and \(\square \in \{1,\text{fin}\}\),
    \[
        \textsf{G}_\square(\Omega_X,\Omega_X)
        \equiv \textsf{G}_\square(\Omega_{X^n},\Omega_{X^n})
        \equiv \textsf{G}_\square(\Omega_{X^{\leq n}},\Omega_{X^{\leq n}})
        \equiv \textsf{G}_\square(\Omega_{X^{<\omega}},\Omega_{X^{<\omega}}).
    \]
\end{lemma}
\begin{proof}
    For the equivalence \(\textsf{G}_\square(\Omega_X,\Omega_X) \equiv \textsf{G}_\square(\Omega_{X^n},\Omega_{X^n})\), we use the map \(\varphi : \mathscr T_X \to \mathscr T_{X^n}\) defined by \(\varphi(U) = U^n\).
    By Lemma \ref{lem:RectangleRefinement}, we know that \(\varphi[\mathscr U] \in \Omega_{X^n}\) given \(\mathscr U \in \Omega_X\).
    Moreover, if \(\mathscr E\) is any collection of open subsets of \(X\) so that \(\varphi[\mathscr E]\) is an \(\omega\)-cover of \(X^n\), it is clear that \(\mathscr E\) must be an \(\omega\)-cover of \(X\).
    Just take \(x \in X\) to the tuple of length \(n\) consisting of \(x\) in each coordinate.

    Observe that \(\Omega_{X^n}\) is closed under enlargement and that, given any \(\omega\)-cover \(\mathscr U\) of \(X^n\), by Lemma \ref{lem:RectangleRefinement},
    \[
        \{ V \in \mathscr T_X : (\exists U \in \mathscr U)[V^n \subseteq U] \} \in \Omega_X.
    \]
    Hence, Corollary \ref{bestCorollary} applies.

    The equivalence \(\textsf{G}_\square(\Omega_X,\Omega_X) \equiv \textsf{G}_\square(\Omega_{X^{\leq n}},\Omega_{X^{\leq n}})\) follows in a similar way, except by using Lemma \ref{lem:TubeRefinement}.

    For the equivalence \(\textsf{G}_\square(\Omega_{X},\Omega_{X}) \equiv \textsf{G}_\square(\Omega_{X^{<\omega}},\Omega_{X^{<\omega}})\), we first note that
    \[
        \textsf{G}_\square(\Omega_{X^{<\omega}},\Omega_{X^{<\omega}}) \leq_{\text{II}} \textsf{G}_\square(\Omega_{X},\Omega_{X})
    \]
    by Lemma \ref{lem:SubspaceMonotone}.
    To obtain
    \[
        \textsf{G}_\square(\Omega_{X},\Omega_{X}) \leq_{\text{II}} \textsf{G}_\square(\Omega_{X^{<\omega}},\Omega_{X^{<\omega}}),
    \]
    we will first need to fix a bijection \(\beta : \omega^2 \to \omega\).
    Though the information transfer across the strategy types is uniform and thus, something similar to one of our translation theorems should apply, we will prove this without referring to them explicitly.

    What we will do is describe how Two is to play the game assuming they have a winning play in \(\textsf{G}_\square(\Omega_{X},\Omega_{X})\).
    Since the statement we wish to prove involves a transferal of winning plays by Two, this will prove what we want.
    Notice that, for \(\{ (n,m) : m \in \omega \}\), Two can play with their attention only on \(X^{\leq n}\).
    In particular, for each \(m \in \omega\), in the \(\beta(n,m)^{\text{th}}\) inning of \(\textsf{G}_\square(\Omega_{X^{<\omega}},\Omega_{X^{<\omega}})\), given One's play \(\mathscr U_{\beta(n,m)}\), let Two choose \(V_{n,m} \subseteq X\) and \(U_{n,m} \in \mathscr U_{\beta(n,m)}\) so that \(V_{n,m}^{\leq n} \subseteq U_{n,m}\) in such a way that \(\{V_{n,m} : m \in \omega \}\) is an \(\omega\)-cover of \(X\).
    This is possible by Lemma \ref{lem:TubeRefinement} and since \(\textsf{G}_\square(\Omega_{X},\Omega_{X}) \equiv \textsf{G}_\square(\Omega_{X^{\leq n}},\Omega_{X^{\leq n}})\).
    Now, the \(U_{n,m}\) correspond to a play by Two in the game \(\textsf{G}_\square(\Omega_{X^{<\omega}},\Omega_{X^{<\omega}})\) and we need only check that it is a winning play.
    For any finite subset \(F\) of \(X^{<\omega}\), there is a maximal length \(n\) of any tuple in \(F\).
    Since \(\{ V^{\leq n}_{n,m} : m \in \omega \}\) is an \(\omega\)-cover of \(X^{\leq n}\) by Lemma \ref{lem:TubeRefinement}, we see that there must be some \(m \in \omega\) for which \(F \subseteq V^{\leq n}_{n,m} \subseteq U_{n,m}\).
    This finishes the proof.
\end{proof}

\begin{lemma} \label{lem:LessThanMenger}
    For any space \(X\) and an ideal \(\mathcal A\) of compact sets so that \(X = \bigcup \mathcal A\),
    \[\textsf{G}_{\text{fin}}(\mathcal O(X,\mathcal A), \mathcal O(X,\mathcal A)) \leq_{\text{II}} \textsf{G}_{\text{fin}}(\mathcal O_X , \mathcal O_X).\]
\end{lemma}
\begin{proof}
    We use Theorem \ref{Theorem16}.
    Note that \(\bigcup \mathcal O (X, \mathcal A) = \bigcup \mathcal O_X = \mathscr T_X\).
    Define \(\Tone : \mathcal O_X \to \mathcal O(X,\mathcal A)\) by the rule
    \[
        \Tone(\mathscr U) = \left\{ \bigcup \mathscr F : \mathscr F \in \left[ \mathscr U \right]^{<\omega} \right\}.
    \]
    Observe that \(\Tone(\mathscr U) \in \mathcal O(X,\mathcal A)\) as \(\mathcal A\) is compact.

    Now we define \(\Ttwo : \left[\mathscr T_X\right]^{<\omega} \times \mathcal O_X \to \left[\mathscr T_X\right]^{<\omega}\) in the following way.
    If \(V_1 , \ldots , V_n \in \mathscr T_X\) and \(\mathscr U \in \mathcal O_X\) are so that \(V_k = \bigcup \mathscr F_k\) for \(\mathscr F_k \in \left[ \mathscr U \right]^{<\omega}\), \(1 \leq k \leq n\), choose \(\mathscr F_{k,\mathscr U} \in \left[ \mathscr U \right]^{<\omega}\) so that \(V_k = \bigcup \mathscr F_{k,\mathscr U}\).
    Then we define
    \[
        \Ttwo\left( \left\{ V_1 , V_2 , \ldots , V_n\right\} , \mathscr U \right) = \bigcup_{k=1}^n \mathscr F_{k,\mathscr U}.
    \]
    Otherwise, let \(\Ttwo\left( \left\{ V_1 , V_2 , \ldots , V_n\right\}, \mathcal O_X \right) = \left\{ V_1 , V_2 , \ldots , V_n\right\}\).

    Suppose \(V_1 , \ldots , V_n \in \Tone(\mathscr U)\).
    Then notice that, for \(1 \leq k \leq n\), \(\mathscr F_{k, \mathscr U} \in \left[ \mathscr U \right]^{<\omega}\) and thus
    \[
        \Ttwo\left( \left\{ V_1 , V_2 , \ldots , V_n\right\} , \mathscr U \right) \in \left[\mathscr U\right]^{<\omega}.
    \]

    Now, suppose \(V_{n,1}, \ldots , V_{n,m_n} \in \Tone(\mathscr U_n)\) are so that \(\bigcup_{n\in\omega} \{ V_{n,j} : 1 \leq j \leq m_n\} \in \mathcal O(X, \mathcal A)\).
    The last thing we need to show is that
    \[
        \bigcup_{n\in\omega}\Ttwo\left(\left\{V_{n,1}, \ldots , V_{n,m_n}\right\} , \mathscr U_n \right) \in \mathcal O_X.
    \]
    Suppose \(x \in X\) and notice that \(\mathcal O(X,\mathcal A) \subseteq \mathcal O_X\).
    Then there exists \(n \in \omega\) and \(1 \leq j \leq m_n\) so that
    \[
        x \in V_{n,j} = \bigcup \mathscr F_{n,j,\mathscr U_n}.
    \]
    Hence, there is an \(W \in \mathscr F_{n,j,\mathscr U_n}\) so that \(x \in W\).
    Finally, note that \(W \in \Ttwo\left(\left\{V_{n,1}, \ldots , V_{n,m_n}\right\}, \mathscr U_n\right)\).
\end{proof}
Lemma \ref{lem:LessThanMenger} can be strengthened to single selections when \(\mathcal A\) is the collection of finite subsets of \(X\).
However, we have not found a way to apply any of the Translation Theorems in this particular instance.
\begin{lemma}[Sakai \cite{Sakai1988}] \label{lem:OmegaRothIsRoth}
    If \(X \models \textsf{S}_{1}(\Omega,\Omega)\), then \(X\) is Rothberger.
\end{lemma}
The proof of this relies on a bijection \(\omega^2 \to \omega\) that ensures that, given a sequence \(\mathscr U_n\) of open covers, single selections from a sequence of \(\omega\)-covers consisting of a particular kind of closure under finite unions of the \(\mathscr U_n\) form single selections from the \(\mathscr U_n\).
The primary reason this creates a problem for strategic transferal is because the way the open covers are translated to \(\omega\)-covers requires the entire sequence up front.
Hence, as Two only knows finitely many of One's moves at any stage in the game, they cannot bring this information to bear.

Also, we employ Pawlikowski's strategy strengthening for the Menger and Rothberger games.
\begin{theorem}[Pawlikowski \cite{Pawlikowski}] \label{thm:Pawlikowski}
    For any space \(X\) and \(\square\in\{1,\text{fin}\}\),
    \[
        \text{I} \underset{\text{pre}}{\uparrow} \textsf{G}_{\square}(\mathcal O_X , \mathcal O_X) \iff \text{I} \uparrow \textsf{G}_{\square}(\mathcal O_X , \mathcal O_X).
    \]
\end{theorem}


\begin{lemma} \label{lem:OmegaMengerLessThanMenger}
    For any space \(X\),
    \[
        \textsf{G}_1(\Omega_X , \Omega_X) \leq_{\text{II}} \textsf{G}_1(\mathcal O_X,\mathcal O_X).
    \]
\end{lemma}
\begin{proof}
    The implication
    \(\text{I} \underset{\text{cnst}}{\not\uparrow} \textsf{G}_1(\Omega_X , \Omega_X) \implies \text{I} \underset{\text{cnst}}{\not\uparrow}\textsf{G}_1(\mathcal O_X,\mathcal O_X)\)
    follows from Lemmas \ref{lem:LessThanMenger} and \ref{lem:Lindelof} and \(\text{I} \underset{\text{pre}}{\uparrow} \textsf{G}_1(\mathcal O_X , \mathcal O_X) \implies \text{I} \underset{\text{pre}}{\uparrow} \textsf{G}_1(\Omega_X , \Omega_X)\)
    is the content of Lemma \ref{lem:OmegaRothIsRoth}.

    Now, using Theorem \ref{thm:Pawlikowski},
    \[
        \text{I} \uparrow \textsf{G}_1(\mathcal O_X , \mathcal O_X)
        \implies \text{I} \underset{\text{pre}}{\uparrow} \textsf{G}_1(\mathcal O_X , \mathcal O_X)
        \implies \text{I} \underset{\text{pre}}{\uparrow} \textsf{G}_1(\Omega_X , \Omega_X)
        \implies \text{I} \uparrow \textsf{G}_1(\Omega_X , \Omega_X).
    \]

    To finish the proof, we note that \cite[Theorem 15]{ClontzHolshouser} states that
    \[
    \text{II} \uparrow \textsf{G}_1(\Omega_X,\Omega_X) \iff \text{II} \uparrow \textsf{G}_1(\mathcal O_X , \mathcal O_X)
    \]
    and that \cite[Theorem 17]{ClontzHolshouser} states that
    \[
    \text{II} \underset{\text{mark}}{\uparrow} \textsf{G}_1(\Omega_X,\Omega_X) \iff \text{II} \underset{\text{mark}}{\uparrow} \textsf{G}_1(\mathcal O_X , \mathcal O_X).
    \]
\end{proof}

\begin{theorem}
    For any space \(X\), \(n \geq 1\), and \(\square \in \{1,\text{fin}\}\),
    \[
        \textsf{G}_\square(\Omega_X,\Omega_X)
        \equiv \textsf{G}_\square(\Omega_{X^n},\Omega_{X^n})
        \equiv \textsf{G}_\square(\Omega_{X^{<\omega}},\Omega_{X^{<\omega}})
        \equiv \textsf{G}_\square(\mathcal O_{X^{<\omega}},\mathcal O_{X^{<\omega}}).
    \]
\end{theorem}
\begin{proof}
    Lemmas \ref{lem:LessThanMenger} and \ref{lem:OmegaMengerLessThanMenger} obtain
    \[
        \textsf{G}_\square(\Omega_{X^{<\omega}},\Omega_{X^{<\omega}}) \leq_{\text{II}} \textsf{G}_\square(\mathcal O_{X^{<\omega}},\mathcal O_{X^{<\omega}}).
    \]
    By Lemma \ref{lem:FirstFinitePower}, to finish the proof, we need only show that
    \[
        \textsf{G}_\square(\mathcal O_{X^{<\omega}},\mathcal O_{X^{<\omega}}) \leq_{\text{II}} \textsf{G}_\square(\Omega_X,\Omega_X).
    \]
    Define \(\varphi : \mathscr T_X \to \mathscr T_{X^{<\omega}}\) by letting \(\varphi(U)\) be the disjoint union of all the \(U^n\), \(n \geq 1\).
    If \(\mathscr U\) is an \(\omega\)-cover of \(X\), observe that \(\varphi[\mathscr U]\) is an open cover of \(X^{<\omega}\).

    Now, suppose \(\mathscr E\) is a collection of open subsets of \(X\) so that \(\varphi[\mathscr E]\) is an open cover of \(X^{<\omega}\).
    To see that \(\mathscr E\) must indeed be an \(\omega\)-cover of \(X\), let \(\{x_1,x_2,\ldots,x_m\} \subseteq X\).
    Then notice that \((x_1,x_2,\ldots , x_m) \in X^m\) so there must be some \(E \in \mathscr E\) so that \((x_1,x_2,\ldots , x_m) \in \varphi(E)\).
    By our definition of \(\varphi\), this means that \(\{x_1,x_2,\ldots,x_n\} \subseteq E\), so we apply Corollary \ref{bestCorollary} to obtain what we claimed.
\end{proof}

\begin{corollary}[Gerlits \& Nagy \cite{GerlitsNagy}] \label{thm:GerlitsNagy}
  Let \(X\) be a space.
  The following are equivalent:
  \begin{enumerate}[label=(\alph*)]
    \item
    \(X\) is an \(\epsilon\)-space,
    \item
    every finite power of \(X\) is an \(\epsilon\)-space,
    \item
    every finite power of \(X\) is Lindel{\"{o}}f, and
    \item
    \(X^{<\omega}\) is Lindel{\"{o}}f.
  \end{enumerate}
\end{corollary}

\begin{corollary}[Just, Miller, and Scheepers {\cite[Theorem 3.9]{JustMillerScheepers}}] \label{thm:JustMillerScheepers}
  Let \(X\) be a space.
  The following are equivalent:
  \begin{enumerate}[label=(\alph*)]
    \item
    \(X \models \textsf{S}_{\text{fin}}(\Omega, \Omega)\),
    \item
    \((\forall n \in\omega)[X^{n+1} \models \textsf{S}_{\text{fin}}(\Omega, \Omega)]\),
    \item
    \((\forall n \in \omega)[X^{n+1} \models \textsf{S}_{\text{fin}}(\mathcal O, \mathcal O)]\), and
    \item
    \(X^{<\omega}\) is Menger.
  \end{enumerate}
\end{corollary}

\begin{corollary}[Sakai \cite{Sakai1988}] \label{thm:Sakai}
  Let \(X\) be a space.
  The following are equivalent:
  \begin{enumerate}[label=(\alph*)]
    \item
    \(X \models \textsf{S}_1(\Omega, \Omega)\),
    \item
    \((\forall n \in\omega)[X^{n+1} \models \textsf{S}_1(\Omega, \Omega)]\),
    \item
    \((\forall n \in \omega)[X^{n+1} \models \textsf{S}_1(\mathcal O, \mathcal O)]\), and
    \item
    \(X^{<\omega}\) is Rothberger.
  \end{enumerate}
\end{corollary}

In the sequel, we will extend Theorem \ref{thm:Pawlikowski} and Corollaries \ref{thm:GerlitsNagy}, \ref{thm:JustMillerScheepers}, and \ref{thm:Sakai} as much as possible.

\section{Applications to Ideals of Compact Sets}

\begin{definition}
    For a space \(X\), let \(\mathbb K(X)\) be the collection of all non-empty compact subsets of \(X\) endowed with the Vietoris topology; that is, the topology generated by sets of the form \(\{ K \in \mathbb K(X) : K \subseteq U \}\) and \(\{ K \in \mathbb K(X) : K \cap U \neq \emptyset\}\) for \(U \subseteq X\) open.
    For \(U_1 , \ldots , U_n\) open in \(X\), define
    \[
        [U_1, \ldots, U_n] = \left\{ K \in \mathbb K(X) : K \subseteq \bigcup_{j=1}^n U_j \text{ and } (\forall j)\left[ K \cap U_j \neq \emptyset \right] \right\}.
    \]
    These sets form a basis for the topology on \(\mathbb K(X)\).
\end{definition}
For a detailed treatment of the Vietoris topology, see \cite{MichaelVietoris}.

\begin{definition}
    We say that an ideal of compact subsets \(\mathcal A^X\) of \(X\) is \emph{closed under \(\mathcal A\)-unions} if there is an ideal \(\mathcal A^{\mathbb K(X)}\) of compact subsets of \(\mathbb A := \mathcal A^X\) as a subspace of \(\mathbb K(X)\) so that
    \begin{itemize}
        \item
        \(K_0 \in \mathcal A^X \implies \{ K \in \mathbb A : K \subseteq K_0\} \in \mathcal A^{\mathbb K(X)}\) and
        \item
        \(\mathbf K \in \mathcal A^{\mathbb K(X)} \implies \bigcup \mathbf K \in \mathcal A^X\).
    \end{itemize}
\end{definition}
\begin{lemma} \label{lem:UnionsClosureIdeals}
    Let \(X\) be a space.
    \begin{itemize}
        \item
        If \(\mathcal A^X\) is the ideal of finite subsets of \(X\), then \(\mathcal A^X\) is closed under \(\mathcal A\)-unions.
        \item
        If \(\mathcal A^X\) is the ideal of compact subsets of \(X\), then \(\mathcal A^X\) is closed under \(\mathcal A\)-unions.
    \end{itemize}
\end{lemma}
\begin{proof}
    For the finite sets, notice that, for any finite set \(F \subseteq X\), \(\mathbf F := \{ K \in \mathbb K(X) : K \subseteq F\}\) consists of finite sets.
    Moreover, \(\mathbf F\) is finite, thus compact.
    Similarly, if \(\mathbf K\) is a finite set consisting of finite subsets of \(X\), then \(\bigcup \mathbf K\) is a finite subset of \(X\).

    The second item follows from results of \cite{MichaelVietoris}.
\end{proof}
Colloquially, one may say that compact sets are closed under compact unions when \(\mathcal A^X\) is the ideal of compact subsets, for example.

The following result finds inspiration from Lemma \ref{lem:RectangleRefinement} and is the foundation for most of what follows.
\begin{lemma} \label{lem:CoverTranslation}
    Let \(X\) be a space and \(\mathcal A^X\) be an ideal of compact subsets that is closed under \(\mathcal A\)-unions.
    \begin{enumerate}[label=(\alph*)]
        \item \label{CoverTranslationBelow}
        If \(\mathscr U \in \mathcal O(X,\mathcal A^X)\), then \(\{[U] : U \in \mathscr U\}\in \mathcal O(\mathbb A, \mathcal A^{\mathbb K(X)})\).
        \item \label{CoverTranslationRefines}
        If \(\mathscr U \in \mathcal O(\mathbb A, \mathcal A^{\mathbb K(X)})\), then
        \[
            \mathscr V = \{ V \in \mathscr T_X : \exists U \in \mathscr U ([V] \subseteq U) \} \in \mathcal O(X, \mathcal A^X).
        \]
    \end{enumerate}
\end{lemma}
\begin{proof}
    \ref{CoverTranslationBelow}
    Suppose \(\mathscr U \in \mathcal O(X,\mathcal A^X)\) and \(\mathbf K \in \mathcal A^{\mathbb K(X)}\).
    Since \(\bigcup \mathbf K \in \mathcal A^X\), there exists \(U \in \mathscr U\) so that \(\bigcup\mathbf K \subseteq U\).
    Hence, \(\mathbf K \subseteq [U]\).

    \ref{CoverTranslationRefines}
    Suppose \(\mathscr U \in \mathcal O(\mathbb A, \mathcal A^{\mathbb K(X)})\) and let \(K_0 \in \mathcal A^X\) be arbitrary.
    Observe that
    \[
    K_0^\ast := \{ K \in \mathbb A : K \subseteq K_0 \} \in \mathcal A^{\mathbb K(X)}.
    \]
    Let \(U \in \mathscr U\) be so that \(K_0^\ast \subseteq U\).
    Now, for each \(K \in K_0^\ast\), we can find a basic neighborhood \(\mathscr B_K\) so that \(K \in \mathscr B_K \subseteq U\).
    By compactness, there are \(K_1 , \ldots , K_n\) so that \(K_0^\ast \subseteq \bigcup_{j=1}^n \mathscr B_{K_j}\).
    Let \(\mathscr B_{K_j} = [ W_{j,1} , \ldots , W_{j,m_j} ]\) for \(1 \leq j \leq n\).
    For \(x \in K_0\), set \(N_x = \bigcap \{ W_{j,k} : x \in W_{j,k} \}\) and define \(V = \bigcup_{x \in K_0} N_x\).
    Clearly, \(K_0 \subseteq V\) and thus \(K_0 \in [V]\), so it suffices to show that \([V] \subseteq U\).

    So let \(K \in [V]\); i.e. \(K \subseteq V\).
    Then we can find \(x_1 , \ldots , x_p \in K_0\) so that \(K \subseteq \bigcup_{\ell=1}^p N_{x_\ell}\) and \(K \cap N_{x_\ell} \neq \emptyset\) for each \(1 \leq \ell \leq p\).
    Since \(\{ x_1 , \ldots , x_p \}\) is a compact subset of \(K_0\), it must be an element of some \([W_{j,1} , \ldots , W_{j,m_j}]\).

    Now, for each \(1 \leq \ell \leq p\), there is a \(q_\ell \leq m_j\) so that \(x_\ell \in W_{j,q_\ell}\); thus \(N_{x_\ell} \subseteq W_{j,q_\ell} \subseteq \bigcup_{q=1}^{m_j} W_{j,q}\).
    So \(K \subseteq \bigcup_{\ell =1}^p N_{x_\ell} \subseteq \bigcup_{q=1}^{m_j} W_{j,q}\).

    For each \(1 \leq q \leq m_j\), let \(1 \leq \ell_q \leq p\) be so that \(x_{\ell_q} \in W_{j,q}\).
    As \(K \cap N_{x_{\ell_q}} \neq \emptyset\) and \(N_{x_{\ell_q}} \subseteq W_{j,q}\), we see that \(K \cap W_{j,q} \neq \emptyset\).

    Hence, we see that \(K \in [W_{j,1} , \ldots , W_{j,m_j}] \subseteq U\).
    Therefore \([V] \subseteq U\).
\end{proof}

\begin{lemma} \label{lem:MengerRothbergerBelow}
    Let \(X\) be a space and \(\mathcal A = \mathcal A^X\) be an ideal of compact subsets that is closed under \(\mathcal A\)-unions where \(\mathbb A = \mathcal A^X\) is viewed as a subspace of \(\mathbb K(X)\).
    Then, for \(\square \in \{ 1 , \text{fin} \}\),
    \[
        \textsf{G}_\square(\mathcal O_{\mathbb A} , \mathcal O_{\mathbb A}) \leq_{\text{II}} \textsf{G}_\square(\mathcal O(X,\mathcal A),\mathcal O(X,\mathcal A)).
    \]
\end{lemma}
\begin{proof}
    Define \(\varphi : \mathscr T_X \to \mathscr T_{\mathbb A}\) by the rule \(\varphi(U) = [U]\).
    Certainly, if \(\mathscr U \in \mathcal O(X,\mathcal A)\), then \(\varphi[\mathscr U] \in \mathcal O_{\mathbb A}\).
    Moreover, suppose \(\mathscr E\) is a collection of open subsets of \(X\) so that \(\varphi[\mathscr E] \in \mathcal O_{\mathbb A}\).
    It is clear that \(\mathscr E \in \mathcal O(X, \mathcal A)\).
    Hence, Corollary \ref{bestCorollary} applies.
\end{proof}
\begin{theorem} \label{thm:HyperspaceEquivalence}
    Let \(X\) be a space and \(\mathcal A = \mathcal A^X\) be an ideal of compact subsets that is closed under \(\mathcal A\)-unions where \(\mathbb A = \mathcal A^X\) is viewed as a subspace of \(\mathbb K(X)\) and \(\mathcal B = \mathcal A^{\mathbb K(X)}\) is the corresponding ideal of \(\mathbb A\).
    Then, for \(\square \in \{ 1 , \text{fin} \}\),
    \[
        \textsf{G}_\square\left(\mathcal O\left(X,\mathcal A\right),\mathcal O\left(X,\mathcal A\right)\right)
        \equiv \textsf{G}_\square\left(\mathcal O\left(\mathbb A,\mathcal B\right),\mathcal O\left(\mathbb A,\mathcal B\right)\right).
    \]
\end{theorem}
\begin{proof}
    Define \(\varphi : \mathscr T_X \to \mathscr T_{\mathbb A}\) by the rule \(\varphi(U) = [U]\).
    If \(\mathscr U \in \mathcal O(X,\mathcal A)\), then, by Lemma \ref{lem:CoverTranslation}, \(\varphi[\mathscr U] \in \mathcal O(\mathbb A,\mathcal B)\).
    If \(\mathscr E\) is a collection of open subsets of \(X\) so that \(\varphi[\mathscr E] \in \mathcal O(\mathbb A, \mathcal B)\), then \(\mathscr E \in \mathcal O(X,\mathcal A)\).
    This is seen by considering any \(K \in \mathcal A\) and noticing that \(K \in \{ F \in \mathbb A : F \subseteq K \} \in \mathcal B\).

    Certainly, \(\mathcal O(\mathbb A,\mathcal B)\) is closed under enlargement and, by Lemma \ref{lem:CoverTranslation}, we see that the \(\varphi\)-refinement of \(\mathscr U \in \mathcal O(\mathbb A, \mathcal B)\) is an element of \(\mathcal O(X,\mathcal A)\).
    Therefore, Corollary \ref{bestCorollary} applies and the proof is complete.
\end{proof}

\subsection{The Space of Finite Subsets} \label{section:Finite}

Let \(\mathcal P_{\text{fin}}(X)\) be the subspace of \(\mathbb K(X)\) consisting of the finite subsets of \(X\).
As in Section \ref{section:Inspiration}, one would anticipate some relationship between open covers of \(\mathcal P_{\text{fin}}(X)\) and \(\omega\)-covers of \(X\).
To some degree, \(\mathcal P_{\text{fin}}(X)\) also avoids all of the unnecessary information the finite powers offer such as repetition and order.
To begin, we offer an analog to Lemma \ref{lem:RectangleRefinement} which follows immediately from Lemmas \ref{lem:UnionsClosureIdeals} and \ref{lem:CoverTranslation}.
\begin{corollary} \label{cor:FiniteBelowRefinement}
    Let \(X\) be a space.
    \begin{enumerate}[label=(\alph*)]
        \item \label{lem:FBR_Belows}
        If \(\mathscr U\) is an \(\omega\)-cover of \(X\), then \(\{[U] : U \in \mathscr U\}\) is an \(\omega\)-cover of \(\mathcal P_{\text{fin}}(X)\).
        \item \label{lem:FBR_Refine}
        If \(\mathscr U\) is an \(\omega\)-cover of \(\mathcal P_{\text{fin}}(X)\), then
        \[
            \mathscr V = \{ V \in \mathscr T_X : \exists U \in \mathscr U ([V] \subseteq U) \}
        \]
        is an \(\omega\)-cover of \(X\).
    \end{enumerate}
\end{corollary}

\begin{theorem} \label{thm:OmegaMenger}
    For any space \(X\) and \(\square \in \{1,\mbox{fin}\}\),
    \[
        \textsf{G}_{\square}(\Omega_X,\Omega_X)
        \equiv \textsf{G}_{\square}(\Omega_{\mathcal P_{\text{fin}}(X)}, \Omega_{\mathcal P_{\text{fin}}(X)})
        \equiv \textsf{G}_{\square}(\mathcal O_{\mathcal P_{\text{fin}}(X)}, \mathcal O_{\mathcal P_{\text{fin}}(X)}).
    \]
\end{theorem}
\begin{proof}
    The fact that
    \[
        \textsf{G}_{\square}(\Omega_X, \Omega_X)
        \equiv \textsf{G}_{\square}(\Omega_{\mathcal P_{\text{fin}}(X)}, \Omega_{\mathcal P_{\text{fin}}(X)})
    \]
    follows from Theorem \ref{thm:HyperspaceEquivalence}.
    By Lemmas \ref{lem:LessThanMenger} and \ref{lem:OmegaMengerLessThanMenger}, we see that
    \[
        \textsf{G}_{\square}(\Omega_{\mathcal P_{\text{fin}}(X)}, \Omega_{\mathcal P_{\text{fin}}(X)})
        \leq_{\text{II}} \textsf{G}_{\square}(\mathcal O_{\mathcal P_{\text{fin}}(X)} , \mathcal O_{\mathcal P_{\text{fin}}(X)}).
    \]
    Finally,
    \[
        \textsf{G}_{\square}(\mathcal O_{\mathcal P_{\text{fin}}(X)} , \mathcal O_{\mathcal P_{\text{fin}}(X)})
        \leq_{\text{II}} \textsf{G}_{\square}(\Omega_X, \Omega_X).
    \]
    follows from Lemma \ref{lem:MengerRothbergerBelow}.
    This finishes the proof.
\end{proof}

\begin{corollary}
    For any space \(X\), the following are equivalent:
    \begin{enumerate}[label=(\alph*)]
        \item \label{epsilonSpaceToFin}
        \(X\) is an \(\epsilon\)-space
        \item \label{epsilonFinToLindelof}
        \(\mathcal P_{\text{fin}}(X)\) is an \(\epsilon\)-space
        \item \label{FinLindelofToEpsilon}
        \(\mathcal P_{\text{fin}}(X)\) is Lindel{\"{o}}f.
    \end{enumerate}
\end{corollary}

\begin{corollary}
    For any space \(X\), the following are equivalent:
    \begin{enumerate}[label=(\alph*)]
        \item
        \(X \models \textsf{S}_{\text{fin}}(\Omega, \Omega)\)
        \item
        \(\mathcal P_{\text{fin}}(X) \models \textsf{S}_{\text{fin}}(\Omega, \Omega)\)
        \item
        \(\mathcal P_{\text{fin}}(X)\) is Menger.
    \end{enumerate}
\end{corollary}

\begin{corollary}\label{cor:OmegaRothberger}
    For any space \(X\), the following are equivalent:
    \begin{enumerate}[label=(\alph*)]
        \item
        \(X \models \textsf{S}_{1}(\Omega, \Omega)\)
        \item
        \(\mathcal P_{\text{fin}}(X) \models \textsf{S}_{1}(\Omega, \Omega)\)
        \item
        \(\mathcal P_{\text{fin}}(X)\) is Rothberger.
    \end{enumerate}
\end{corollary}

\begin{corollary}[Scheepers \cite{Scheepers}]
    For any space \(X\) and \(\square\in\{1,\text{fin}\}\),
    \[
        \text{I} \underset{\text{pre}}{\uparrow} \textsf{G}_{\square}(\Omega_X,\Omega_X)
        \iff \text{I} \uparrow \textsf{G}_{\square}(\Omega_X,\Omega_X).
    \]
\end{corollary}
\begin{proof}
    These follow immediately from Theorems \ref{thm:Pawlikowski} and \ref{thm:OmegaMenger}.
\end{proof}

\subsection{The Space of Compact Subsets} \label{sec:Compacts}

When we move to compact sets, one might expect that all of the analogous theorems from Sections \ref{section:Inspiration} and \ref{section:Finite} hold between \(k\)-covers of \(X\) and open covers of \(\mathbb K(X)\).
Though we cannot obtain the full scope of those results, we are able recover a significant fragment; namely, everything about finite selection games goes through, and we are able to recover a version of Pawlikowki's theorem.

In a similar spirit to the results of Lemma \ref{lem:RectangleRefinement} and Corollary \ref{cor:FiniteBelowRefinement}, we establish a way to transfer \(k\)-cover information between \(X\) and \(\mathbb K(X)\).
It follows immediately from Lemmas \ref{lem:UnionsClosureIdeals} and \ref{lem:CoverTranslation}.
\begin{corollary} \label{cor:KBelowRefinement}
    Let \(X\) be a space.
    \begin{enumerate}[label=(\alph*)]
        \item \label{lem:KBR_Belows}
        If \(\mathscr U\) is a \(k\)-cover of \(X\), then \(\{[U] : U \in \mathscr U\}\) is a \(k\)-cover of \(\mathbb K(X)\).
        \item \label{lem:KBR_Refine}
        If \(\mathscr U\) is a \(k\)-cover of \(\mathbb K(X)\), then
        \[
            \mathscr V = \{ V \in \mathscr T_X : \exists U \in \mathscr U ([V] \subseteq U) \}
        \]
        is a \(k\)-cover of \(X\).
    \end{enumerate}
\end{corollary}

\begin{corollary} \label{cor:KEquivalence}
    For any space \(X\) and \(\square \in \{ 1 , \text{fin} \}\),
    \[
        \textsf{G}_{\square}(\mathcal O_{\mathbb K(X)}, \mathcal O_{\mathbb K(X)})
        \leq_{\text{II}} \textsf{G}_{\square}(\mathcal K_X, \mathcal K_X)
        \equiv \textsf{G}_{\square}(\mathcal K_{\mathbb K(X)}, \mathcal K_{\mathbb K(X)}).
    \]
\end{corollary}
\begin{proof}
    This follows immediately from Lemma \ref{lem:MengerRothbergerBelow} and Theorem \ref{thm:HyperspaceEquivalence}.
\end{proof}

\begin{theorem} \label{thm:KMenger}
    For any space \(X\),
    \[
        \textsf{G}_{\text{fin}}(\mathcal K_X, \mathcal K_X)
        \equiv \textsf{G}_{\text{fin}}(\mathcal K_{\mathbb K(X)}, \mathcal K_{\mathbb K(X)})
        \equiv \textsf{G}_{\text{fin}}(\mathcal O_{\mathbb K(X)}, \mathcal O_{\mathbb K(X)}).
    \]
\end{theorem}
\begin{proof}
    This follows immediately from Lemma \ref{lem:LessThanMenger} and Corollary \ref{cor:KEquivalence}.
\end{proof}

\begin{corollary}
    For any space \(X\), the following are equivalent:
    \begin{enumerate}[label=(\alph*)]
        \item
        \(X\) is \(k\)-Lindel{\"{o}}f
        \item
        \(\mathbb K(X)\) is \(k\)-Lindel{\"{o}}f
        \item
        \(\mathbb K(X)\) is Lindel{\"{o}}f.
    \end{enumerate}
\end{corollary}
\begin{corollary}
    For any space \(X\), the following are equivalent:
    \begin{enumerate}[label=(\alph*)]
        \item
        \(X \models \textsf{S}_{\text{fin}}(\mathcal K, \mathcal K)\)
        \item
        \(\mathbb K(X) \models \textsf{S}_{\text{fin}}(\mathcal K, \mathcal K)\)
        \item
        \(\mathbb K(X)\) is Menger.
    \end{enumerate}
\end{corollary}

\begin{corollary}
    For any space \(X\), \[\text{I} \underset{\text{pre}}{\uparrow} \textsf{G}_{\text{fin}}(\mathcal K_X , \mathcal K_X) \iff \text{I}\uparrow \textsf{G}_{\text{fin}}(\mathcal K_X , \mathcal K_X).\]
\end{corollary}
\begin{proof}
    This follows from Theorems \ref{thm:Pawlikowski} and \ref{thm:KMenger}.
\end{proof}

Like before, single selections present an obstacle.
In the context of \(k\)-covers, we don't even obtain an analog to Corollary \ref{cor:OmegaRothberger}.
\begin{example}
    In general, \(\textsf{S}_1(\mathcal K, \mathcal K) \not\Rightarrow \textsf{S}_1(\mathcal O , \mathcal O)\).
    Observe that \(\mathbb R \models \textsf{S}_1(\mathcal K, \mathcal K)\) but \(\mathbb R \not\models \textsf{S}_1(\mathcal O,\mathcal O)\).
    If \(\{ \mathscr U_n : n \in \omega \}\) consists of \(k\)-covers of \(\mathbb R\), simply choose \(U_n \in \mathscr U_n\) to be so that \([-n,n] \subseteq U_n\).
    Then \(\{ U_n : n \in \omega\}\) is a \(k\)-cover of \(\mathbb R\).
    On the other hand, consider \(\mathscr V_n = \{ B(q, 2^{-n}) : q \in \mathbb Q\}\) and any sequence of selections \(V_n \in \mathscr V_n\).
    Notice that the union of the \(V_n\) has finite Lebesgue measure so they cannot cover \(\mathbb R\).
\end{example}

Because of this non-example, we cannot obtain a version of Pawlikowski's result for \(k\)-covers as easily as we did for \(\omega\)-covers.
In the next two results we nevertheless prove that there is a Pawlikowski style strategy reduction for the \(k\)-Rothberger game.
The basic idea is to take the game up to the hyperspace and play with the right kind of open sets to guarantee that Two's play results in a \(k\)-cover.

\begin{lemma} \label{lem:AscendingSelections}
    Suppose \(\mathcal A\) is any ideal of closed sets that contains all singletons of a space \(X\).
    Also suppose \(\bigcup\{ \mathscr F_n : n \in \omega\} \in \mathcal O(X,\mathcal A)\) where \(\mathscr F_n\) is a finite collection of open sets for each \(n \in \omega\).
    Then, for any \(A \in \mathcal A\), there exists an increasing sequence \(\{\alpha_n : n \in \omega\}\) so that
    \[
        (\forall n \in \omega)(\exists U \in \mathscr F_{\alpha_n})\left[ A \subseteq U \right].
    \]
\end{lemma}
\begin{proof}
    Let \(A = A_0 \in \mathcal A\) be arbitrary and, for \(n \geq 0\), suppose we have \(A_n \in \mathcal A\) and \(\alpha_n \in \omega\) defined so that
    \[
        \alpha_n = \min\left\{ m \in \omega : \left(\exists U \in \mathscr F_{m}\right)\left[ A_n \subseteq U \right] \right\}.
    \]
    As \(U\) is a proper open set for each \(U \in \mathscr F_{\alpha_n}\), we can find \(x_U \in X \setminus U\).
    Notice that
    \[
        A_{n+1} := A_n \cup \{ x_U : U \in \mathscr F_{\alpha_n} \} \in \mathcal A
    \]
    since \(\mathcal A\) is an ideal containing singletons and \(\mathscr F_{\alpha_n}\) is finite.
    So then we can set
    \[
        \alpha_{n+1} = \min\left\{ m \in \omega : \left(\exists U \in \mathscr F_{m}\right)\left[ A_{n+1} \subseteq U \right] \right\}.
    \]
    Observe that \(\alpha_{n+1} > \alpha_n\).
    This finishes the proof.
\end{proof}

\begin{theorem}
    For any space \(X\), \(\text{I} \underset{\text{pre}}{\uparrow} \textsf{G}_1(\mathcal K_X , \mathcal K_X) \iff \text{I}\uparrow \textsf{G}_1(\mathcal K_X , \mathcal K_X)\).
\end{theorem}
\begin{proof}
    We need only show
    \[
        \text{I} \underset{\text{pre}}{\not\uparrow} \textsf{G}_1(\mathcal K_X , \mathcal K_X) \implies \text{I} \not\uparrow \textsf{G}_1(\mathcal K_X , \mathcal K_X).
    \]
    Suppose \(X \models \textsf{S}_1(\mathcal K_X, \mathcal K_X)\) and that One is playing according to some fixed strategy in \(\textsf{G}_1(\mathcal K_X , \mathcal K_X)\).
    Any \(k\)-cover can be made into a countable \(k\)-cover by the selection principle.
    Hence, we can code One's strategy with \(\{ U_s : s \in \omega^{<\omega}\}\) with the property that \(\{U_{s \concat k} : k \in \omega\}\) is a \(k\)-cover for any \(s \in \omega^{<\omega}\).
    Using this strategy for One in \(\textsf{G}_1(\mathcal K_X , \mathcal K_X)\), we will define a strategy for One in \(\textsf{G}_{\text{fin}}(\mathcal K_{\mathbb K(X)} , \mathcal K_{\mathbb K(X)})\) which will produce a winning counter-play by Two as
    \begin{align*}
        X \models \textsf{S}_1(\mathcal K_X, \mathcal K_X)
        &\implies \mathbb K(X) \models \textsf{S}_1(\mathcal K_{\mathbb K(X)}, \mathcal K_{\mathbb K(X)})\\
        &\implies \mathbb K(X) \models \textsf{S}_{\text{fin}}(\mathcal K_{\mathbb K(X)}, \mathcal K_{\mathbb K(X)})\\
        &\implies \text{I} \not\uparrow \textsf{G}_{\text{fin}}(\mathcal K_{\mathbb K(X)}, \mathcal K_{\mathbb K(X)}).
    \end{align*}
    Moreover, we will show the counter-play produced actually corresponds to a counter-play to One's strategy in \(\textsf{G}_1(\mathcal K_X , \mathcal K_X)\).
    We will do this through a sequence of useful claims.

    The first claim is that
    \[
        (\forall m \geq 0)(\forall j > 0)(\forall \mathbf K \in \mathbb K(\mathbb K(X)))(\exists s \in \omega^{|j^m|})(\forall t \in j^m)(\forall L \in \mathbf K)(\exists k < |j^m|)\left[ L \subseteq U_{t \concat (s \restriction_{k+1})} \right].
    \]
    Fix \(m \geq 0\), \(j > 0\), let \(\mathbf K \subseteq \mathbb K(X)\) be compact, and consider \(K := \bigcup \mathbf K\), which forms a compact subset of \(X\).
    Enumerate \(j^m\) as \(\{ t_\ell : \ell < |j^m| \}\).
    Let \(s(0) \in \omega\) so that \(K \subseteq U_{t_0 \concat s(0)}\).
    Then, for \(n \geq 0\), suppose we have \(s(0), \ldots, s(n) \in \omega\) defined so that \(K \subseteq U_{t_\ell \concat s(0) \concat \cdots \concat s(\ell)}\) for each \(\ell \leq n\).
    Let \(s(n+1) \in \omega\) be so that \[K \subseteq U_{t_{n+1} \concat s(0) \concat s(1) \concat \cdots \concat s(n+1)}.\]
    This defines \(s : |j^m| \to \omega\).

    Next, fix some \(t \in j^m\), let \(L \in \mathbf K\) be arbitrary, and find \(k < |j^m|\) so that \(t = t_k\).
    Then,
    \[
        L \subseteq K \subseteq U_{t_k \concat s(0) \concat \cdots \concat s(k)} = U_{t \concat s\restriction_{k+1}}.
    \]
    This establishes the claim.

    The second claim involves defining, for \(m \geq 0\), \(j > 0\), and \(s : |j^m| \to \omega\),
    \[
        \mathbf V_s(m,j) = \bigcap_{t \in j^m} \bigcup_{k=1}^{|j^m|}  [U_{t \concat (s \restriction_k)}].
    \]
    The second claim is that, for fixed \(m \geq 0\) and \(j > 0\), \(\left\{ \mathbf V_s(m,j) : s \in \omega^{|j^m|} \right\}\) is a \(k\)-cover of \(\mathbb K(X)\).
    So let \(\mathbf K \subseteq \mathbb K(X)\) be compact and choose \(s : |j^m| \to \omega\) so that
    \[
        (\forall t \in j^m)(\forall L \in \mathbf K)(\exists k < |j^m|)\left[ L \subseteq U_{t \concat (s \restriction_{k+1})} \right],
    \]
    which is guaranteed by the first claim.
    Fix \(t \in j^m\), let \(L \in \mathbf K\), and observe that, for some \(k < |j^m|\), \[L \in [U_{t \concat (s\restriction_{k+1})}] \subseteq \bigcup_{\ell=1}^{|j^m|} [U_{t \concat (s\restriction_\ell)}].\]
    Since this is true for any \(L \in \mathbf K\), we see that
    \[
        \mathbf K \subseteq \bigcup_{\ell=1}^{|j^m|} [U_{t \concat (s\restriction_\ell)}].
    \]
    Since \(t \in j^m\) was also taken to be arbitrary, we see that
    \[
        \mathbf K \subseteq \bigcap_{t \in j^m}\bigcup_{k=1}^{|j^m|} [U_{t \concat (s\restriction_k)}] = \mathbf V_s(m,j).
    \]

    The third claim is that there are increasing functions \(g,h : \omega \to \omega\) so that
    \[
        (\forall \mathbf K \in \mathbb K(\mathbb K(X)))(\exists^\infty n \in \omega)(\exists s : (g(n+1) - g(n)) \to h(n+1))\left[ \mathbf K \subseteq \mathbf V_s(g(n), h(n)) \right].
    \]
    To accomplish this, we define a particular strategy \(\sigma\) for One in \(\textsf{G}_{\text{fin}}(\mathcal K_{\mathbb K(X)} , \mathcal K_{\mathbb K(X)})\).
    First, for \(m \geq 0\) and \(j > 0\), define
    \[
        \mathscr V_{m,j} = \left\{ \mathbf V_s(m,j) : s \in \omega^{|j^m|} \right\},
    \]
    which is a \(k\)-cover of \(\mathbb K(X)\) by the second claim.
    Also, for \(m \geq 0\), \(j > 0\), and \(p > 0\), define
    \[
        \mathscr F_{m,j,p} = \left\{ \mathbf V_s(m,j) : s \in p^{|j^m|}\right\},
    \]
    a finite subset of \(\mathscr V_{m,j}\).
    Observe that, for \(m \geq 0\) and \(j > 0\), if \(0 < p \leq q\), then \(\mathscr F_{m,j,p} \subseteq \mathscr F_{m,j,q}.\)
    In fact,
    \[
        \left(\forall \mathscr E \in [\mathscr V_{m,j}]^{<\omega}\right)\left( \exists p > 0 \right)\left[ \mathscr E \subseteq \mathscr F_{m,j,p} \right].
    \]
    To see this, let \(\mathscr E \in [\mathscr V_{m,j}]^{<\omega}\) and \(A \in \left[ \omega^{|j^m|} \right]^{<\omega}\) be so that \(\mathscr E = \{ \mathbf V_s(m,j) : s \in A \}\).
    Then set
    \[
        p = 1 + \max\{ s(\ell) : \ell < |j^m|, s \in A\}.
    \]
    It follows that \(A \subseteq p^{|j^m|}\) which further implies that \(\mathscr E \subseteq \mathscr F_{m,j,p}\).
    Now define \(p_{m,j} : \left[ \mathscr V_{m,j} \right]^{<\omega} \to \omega\) to be
    \[
        p_{m,j}(\mathscr E) = \min\{ p : \mathscr E \subseteq \mathscr F_{m,j,p} \}.
    \]

    We next define the strategy \(\sigma\).
    Set \(m_0 = 0\), \(j_0 = 1\), and \(\sigma(\emptyset) = \mathscr V_{m_0,j_0}\).
    For \(n \geq 0\), suppose \(\{ \mathscr E_\ell : \ell < n \}\), \(\{m_\ell:\ell \leq n\}\), and \(\{j_\ell:\ell \leq n\}\) have been defined.
    Let \(m_{n+1} = m_n + |j_n^{m_n}|\).
    Then for \(\mathscr E_n \in \left[\mathscr V_{m_n,j_n}\right]^{<\omega}\) set \(j_{n+1} = \max\{ j_n , p_{m_n,j_n}(\mathscr E_n)\}\) and define
    \[
        \sigma(\mathscr V_{m_0,j_0} , \mathscr E_0 , \ldots , \mathscr V_{m_n,j_n} , \mathscr E_n) = \mathscr V_{m_{n+1},j_{n+1}}
    \]
    This finishes the definition of the strategy \(\sigma\).

    As One does not have a winning strategy in \(\textsf{G}_{\text{fin}}(\mathcal K_{\mathbb K(X)} , \mathcal K_{\mathbb K(X)})\), Two can produce a counter-play \(\{ \mathscr E_n : n \in \omega \}\) so that
    \[
        \bigcup \{ \mathscr E_n : n \in \omega \}
    \]
    is a \(k\)-cover of \(\mathbb K(X)\).
    Notice that this provides increasing sequences \(\langle m_n : n \in \omega \rangle\) and \(\langle j_n : n \in \omega \rangle\).
    Moreover, as \(\mathscr E_n \in \left[ \mathscr V_{m_n,j_n} \right]^{<\omega}\) and \(j_{n+1} \geq p_{m_n,j_n}(\mathscr E_n)\), we have that \(\mathscr E_n \subseteq \mathscr F_{m_n,j_n,j_{n+1}}\).
    That is,
    \[
        \bigcup \{ \mathscr F_{m_n,j_n,j_{n+1}} : n \in \omega \}
    \]
    is a \(k\)-cover of \(\mathbb K(X)\).

    Define \(g,h : \omega \to \omega\) by the rules \(g(n) = m_n\) and \(h(n) = j_n\) and notice that they are increasing functions.
    To verify they are as desired, we first find one \(n\in\omega\) that meets the requisite criterion.
    Let \(\mathbf K \subseteq \mathbb K(X)\) be compact.
    Since \(\bigcup \{ \mathscr F_{m_n,j_n,j_{n+1}} : n \in \omega \}\) is a \(k\)-cover of \(\mathbb K(X)\), there must be some \(n \in \omega\) and \(s : |j_n^{m_n}| \to j_{n+1}\) so that \(\mathbf K \subseteq \mathbf V_s(m_n,j_n)\).
    Behold that \(m_n = g(n)\), \(j_n = h(n)\), \(j_{n+1} = h(n+1)\), and \(|j_n^{m_n}| = m_{n+1} - m_n = g(n+1)-g(n)\).
    The fact that infinitely many such \(n\) exist follows from Lemma \ref{lem:AscendingSelections}.

    The final thing to show is that we can actually construct a counter-play against One's strategy in \(\textsf{G}_1(\mathcal K_X , \mathcal K_X)\) with the help of the defined \(g\) and \(h\).
    For \(n \geq 1\), \(k_1 < k_2 < \cdots < k_n\), and \(s_i: (g(k_i + 1) - g(k_i)) \to h(k_i + 1)\), \(1 \leq i \leq n\), we define
    \[
        \mathbf W_n(k_1,\ldots,k_n;s_1,\ldots,s_n) = \bigcap_{i=1}^n \mathbf V_{s_i}(g(k_i),h(k_i)).
    \]
    To assist with notation, we let \({F}_{k_i} = h(k_i + 1)^{(g(k_i + 1) - g(k_i))}\).
    By the third claim,
    \[
        \mathscr W_n :=
        \left\{ \mathbf W_n(k_1,\ldots, k_n ; s_1 , \ldots , s_n) : (k_1 < \cdots < k_n) \text{ and }(\forall 1 \leq i \leq n)\left[s_i \in {F}_{k_i}\right] \right\}
    \]
    is a \(k\)-cover of \(\mathbb K(X)\).
    Since we are assume \(X \models \textsf{S}_1(\mathcal K_X, \mathcal K_X)\) and we know that
    \[
        X \models \textsf{S}_1(\mathcal K_X, \mathcal K_X) \implies \mathbb K(X) \models \textsf{S}_1(\mathcal K_{\mathbb K(X)}, \mathcal K_{\mathbb K(X)}),
    \]
    for each \(n \geq 1\), we can select \(k_{n,1}<\ldots<k_{n,n}\) and \(s_{n,i} \in {F}_{k_{n,i}}\) for \(1 \leq i \leq n\), so that
    \[
        \{ \mathbf W_n(k_{n,1}, \ldots, k_{n,n};s_{n,1}, \ldots , s_{n,n}) : n \in \omega\}
    \]
    is a \(k\)-cover of \(\mathbb K(X)\).

    For each \(n \geq 1\), choose \(k_{n,\alpha_n} \in \{k_{n,1} , \ldots , k_{n,n}\} \setminus \{ k_{\ell, \alpha_\ell} : 1 \leq \ell < n\}\) and consider
    \[
        B_n := \left\{g(k_{n,\alpha_n}) + i : i < g(k_{n,\alpha_n}+1) - g(k_{n,\alpha_n}) \right\}.
    \]
    We argue that the sets \(\{ B_n : n \geq 1 \}\) are pair-wise disjoint.
    Suppose \(m < n\) and notice that \(k_{m,\alpha_m} \neq k_{n,\alpha_n}\) by definition.
    Without loss of generality, suppose \(k_{m,\alpha_m} < k_{n,\alpha_n}\).
    Since it follows that \(g(k_{m,\alpha_m}) < g(k_{n,\alpha_n})\), to establish that  \(B_m\) and \(B_n\) are dsijoint, it suffices to check that \(g(k_{m,\alpha_m}+1) \leq g(k_{n,\alpha_n})\).
    Indeed,
    \begin{align*}
        k_{m,\alpha_m} + 1 \leq k_{n,\alpha_n}
        &\implies g(k_{m,\alpha_m} + 1) \leq g(k_{n,\alpha_n}).
    \end{align*}
    Moreover, for any \(\ell \in B_n\), \(\ell - g(k_{n,\alpha_n}) < g(k_{n,\alpha_n}+1) - g(k_{n,\alpha_n})\). So \(s_{n,\alpha_n}(\ell - g(k_{n,\alpha_n}))\) is defined.
    This allows us to define \(f : \omega \to \omega\) by the rule
    \[
        f(\ell)
        = \begin{cases}
            s_{n,\alpha_n}(\ell - g(k_{n,\alpha_n})), & \ell \in B_n\\
            0, & \text{otherwise}
        \end{cases}
    \]

    We claim that this \(f\) is Two's desired play.
    Let \(K \subseteq X\) be compact, and notice that \(\{K\}\) is a compact subset of \(\mathbb K(X)\).
    So there exists some \(n \geq 1\) so that
    \[
        \{K\} \subseteq \mathbf W_n(k_{n,1}, \ldots, k_{n,n};s_{n,1}, \ldots , s_{n,n})
        = \bigcap_{i=1}^n \mathbf V_{s_{n,i}}(g(k_{n,i}),h(k_{n,i})).
    \]
    For ease of notation, let \(E_n = h(k_{n,\alpha_n})^{g(k_{n,\alpha_n})}\) and notice that
    \[
        \{K\} \subseteq \mathbf V_{s_{n,\alpha_n}}(g(k_{n,\alpha_n}),h(k_{n,\alpha_n}))
        = \bigcap_{t \in E_n} \bigcup_{\ell=1}^{|E_n|} [U_{t \concat (s_{n,\alpha_n} \restriction_\ell)}].
    \]
    We wish to show that \(f\restriction_{g(k_{n,\alpha_n})} : g(k_{n,\alpha_n}) \to h(k_{n,\alpha_n})\).
    So let \(\ell < g(k_{n,\alpha_n})\) be arbitrary.
    If \(\ell \not\in B_m\) for any \(m \geq 1\), \(f(\ell) = 0 < h(k_{n,\alpha_n})\).
    Otherwise, \(\ell \in B_m\) for some \(m \geq 1\).
    Then there is some \(i < g(k_{m,\alpha_m}+1) - g(k_{m,\alpha_m})\) with \(\ell = g(k_{m,\alpha_m}) + i\).
    Hence,
    \[
        f(\ell) = s_{m,\alpha_m}(\ell - g(k_{m,\alpha_m}))
        = s_{m,\alpha_m}(i) < h(k_{m,\alpha_m}).
    \]
    Also, as
    \[
        g(k_{m,\alpha_m}) \leq g(k_{m,\alpha_m}) + i = \ell < g(k_{n,\alpha_n}),
    \]
    we see that \(k_{m,\alpha_m} < k_{n,\alpha_n}\) which provides
    \[
        f(\ell) < h(k_{m,\alpha_m}) \leq h(k_{n,\alpha_n}).
    \]
    Thus,
    \[
        \{K\} \subseteq \bigcup_{\ell=1}^{|E_n|} [U_{(f\restriction_{g(k_{n,\alpha_n})}) \concat (s_{n,\alpha_n} \restriction_\ell)}]
        \implies K \in \bigcup_{\ell=1}^{|E_n|} [U_{(f\restriction_{g(k_{n,\alpha_n})}) \concat (s_{n,\alpha_n} \restriction_\ell)}]
    \]
    which means that for some \(1 \leq \ell \leq |E_n|\),
    \[
        K \in [U_{(f\restriction_{g(k_{n,\alpha_n})}) \concat (s_{n,\alpha_n} \restriction_\ell)}] \implies K \subseteq U_{(f\restriction_{g(k_{n,\alpha_n})}) \concat (s_{n,\alpha_n} \restriction_\ell)}.
    \]
    Finally, by our definition of \(f\), we note that
    \[
        (f\restriction_{g(k_{n,\alpha_n})}) \concat (s_{n,\alpha_n} \restriction_\ell)
        = f \restriction_{g(k_{n,\alpha_n}+1)}
    \]
    which means
    \[
        K \subseteq U_{f \restriction_{g(k_{n,\alpha_n}+1)}}.
    \]
    Therefore, if Two plays according to \(f\), Two produces a \(k\)-cover of \(X\), finishing the proof.
\end{proof}

\section{Final Remarks}

We end with a couple of other applications of these techniques which relate to the interplay between cover types.
\begin{theorem} \label{thm:KOmega}
    For any space \(X\) and \(\square \in \{1,\text{fin}\}\),
    \[
        \textsf{G}_\square(\mathcal K_{\mathbb K(X)},\Omega_{\mathbb K(X)}) \leq_{\text{II}} \textsf{G}_\square(\mathcal K_{X},\Omega_{X}).
    \]
\end{theorem}
\begin{proof}
    Define \(\varphi : \mathscr T_X \to \mathscr T_{\mathbb K(X)}\) by the rule \(\varphi(U) = [U]\).
    By Corollary \ref{cor:KBelowRefinement}, we know that \(\varphi[\mathscr U] \in \mathcal K_{\mathbb K(X)}\) when \(\mathscr U \in \mathcal K_X\).
    Now, suppose \(\mathscr E \subseteq \mathscr T_X\) is so that \(\varphi[\mathscr E] \in \Omega_{\mathbb K(X)}\).
    Observe that \(\varphi[\mathscr E]\) is also an \(\omega\)-cover of \(\mathcal P_{\text{fin}}(X)\).
    By Corollary \ref{cor:FiniteBelowRefinement}, we know that
    \[
        \left\{ V \in \mathscr T_X : \exists U \in \varphi[\mathscr E]\left([V] \subseteq U\right) \right\} \in \Omega_X
    \]
    which demonstrates that \(\mathscr E\) is an \(\omega\)-cover of \(X\).
    Thus, Corollary \ref{bestCorollary} applies.
\end{proof}

\begin{theorem}
    For any space \(X\) and \(\square \in \{1,\text{fin}\}\),
    \[
        \textsf{G}_\square(\Omega_X , \mathcal K_X) \leq_{\text{II}} \textsf{G}_\square(\Omega_{\mathbb K(X)}, \mathcal K_{\mathbb K(X)}).
    \]
\end{theorem}
\begin{proof}
    In this case, we apply Corollary \ref{Corollary17}.
    Define \(\Tone: \Omega_{\mathbb K(X)} \to \Omega_X\) by the rule
    \[
        \Tone(\mathscr U) = \left\{ V \in \mathscr T_X : \exists U \in \mathscr U\left( [V] \subseteq U \right) \right\}.
    \]
    Since every \(\omega\)-cover of \(\mathbb K(X)\) is an \(\omega\)-cover of \(\mathcal P_{\text{fin}}(X)\), \(\Tone\) is defined by Corollary \ref{cor:FiniteBelowRefinement}.

    Now we define \(\Ttwo : \mathscr T_X \times \Omega_{\mathbb K(X)} \to \mathscr T_{\mathbb K(X)}\) in the following way.
    Let \(\mathscr U \in \Omega_{\mathbb K(X)}\) and if \(V \in \Tone(\mathscr U)\), let \(\Ttwo(V, \mathscr U)\) be so that \([V] \subseteq \Ttwo(V,\mathscr U)\).
    Otherwise, let \(\Ttwo(V,\mathscr U) = V\).

    Suppose \(\mathscr F_n \in \left[ \Tone(\mathscr U_n) \right]^{<\omega}\) are so that \(\bigcup_{n\in\omega}\mathscr F_n \in \mathbb K(X)\).
    By Corollary \ref{cor:KBelowRefinement}, we know that
    \[
        \bigcup_{n\in\omega} \left\{ [V] : V \in \mathscr F_n \right\} \in \mathcal K_{\mathbb K(X)}
    \]
    and, as \([V] \subseteq \Ttwo(V,\mathscr U_n)\) for each \(V \in \mathscr F_n\), we see that
    \[
        \bigcup_{n\in\omega} \left\{ \Ttwo(V,\mathscr U_n) : V \in \mathscr F_n \right\} \in \mathcal K_{\mathbb K(X)}.
    \]
    This finishes the proof.
\end{proof}

For further work, are \cite[Lemma 7]{CHContinuousFunctions} and \cite[Lemma 8]{CHContinuousFunctions} true as stated?
Additionally, can Theorem \ref{thm:KOmega} be used to establish a Pawlikowski style strategy reduction for \(\textsf{G}_\square(\mathcal K, \Omega)\)?

\providecommand{\bysame}{\leavevmode\hbox to3em{\hrulefill}\thinspace}
\providecommand{\MR}{\relax\ifhmode\unskip\space\fi MR }
\providecommand{\MRhref}[2]{%
  \href{http://www.ams.org/mathscinet-getitem?mr=#1}{#2}
}
\providecommand{\href}[2]{#2}


\begin{thebibliography}{10}

\bibitem{CHCompactOpen}
Christopher Caruvana and Jared Holshouser, \emph{Closed discrete selection in
  the compact open topology}, Topology Proceedings \textbf{56} (2020), 25--55.

\bibitem{CHContinuousFunctions}
\bysame, \emph{Selection games on continuous functions}, Topology and its
  Applications \textbf{279} (2020), 107253.

\bibitem{CHHyperspaces}
\bysame, \emph{Selection games on hyperspaces}, arXiv:2012.06668, 2020.

\bibitem{Caserta}
A.~Caserta, G.~{Di Maio}, Lj.D.R. Ko{\v{c}}inac, and E.~Meccariello,
  \emph{Applications of k-covers {II}}, Topology and its Applications
  \textbf{153} (2006), no.~17, 3277--3293, Special Issue: Topology and Analysis
  in Applications.

\bibitem{ClontzDuality}
Steven Clontz, \emph{Dual selection games}, Topology and its Applications
  \textbf{272} (2020), 107056.

\bibitem{ClontzHolshouser}
Steven Clontz and Jared Holshouser, \emph{Limited information strategies and
  discrete selectivity}, Topology and its Applications \textbf{265} (2019),
  106815.

\bibitem{KocinacEtAl2005}
G.~{Di Maio}, Lj.D.R. Ko{\v{c}}inac, and E.~Meccariello, \emph{Selection
  principles and hyperspace topologies}, Topology and its Applications
  \textbf{153} (2005), no.~5, 912--923, The Special Issue: The Fifth
  Iberoamerican Conference on General Topology and its Applications (V CITA).

\bibitem{GerlitsNagy}
J.~Gerlits and Zs. Nagy, \emph{Some properties of {$C(X)$}, {I}}, Topology and
  its Applications \textbf{14} (1982), no.~2, 151 -- 161.

\bibitem{JustMillerScheepers}
Winfried Just, Arnold~W. Miller, Marion Scheepers, and Paul~J. Szeptycki,
  \emph{The combinatorics of open covers {II}}, Topology and its Applications
  \textbf{73} (1996), no.~3, 241 -- 266.

\bibitem{KocinacSelectedResults}
Lj.D.R. Ko{\v{c}}inac, \emph{Selected results on selection principles},
  Proceedings of the Third Seminar on Geometry and Topology (Tabriz, Iran),
  July 15-17, 2004, pp.~71--104.

\bibitem{Li2016}
Zuquan Li, \emph{Selection principles of the {F}ell topology and the {V}ietoris
  topology}, Topology and its Applications \textbf{212} (2016), 90--104.

\bibitem{MichaelVietoris}
Ernest Michael, \emph{Topologies on spaces of subsets}, Transactions of the
  American Mathematical Society \textbf{71} (1951), 152--182.

\bibitem{Mrsevic}
Mila Mr{\v{s}}evi{\'{c}} and Milena Jeli{\'{c}}, \emph{Selection principles in
  hyperspaces with generalized {V}ietoris topologies}, Topology and its
  Applications \textbf{156} (2008), no.~1, 124--129, The Third Workshop on
  Coverings, Selections and Games in Topology.

\bibitem{Osipov}
Alexander~V. Osipov, \emph{Selectors for sequences of subsets of hyperspaces},
  Topology and its Applications \textbf{275} (2020), 107007.

\bibitem{Pawlikowski}
Janusz Pawlikowski, \emph{Undetermined sets of point-open games}, Fundamenta
  Mathematicae \textbf{144} (1994), no.~3, 279--285.

\bibitem{Sakai1988}
Masami Sakai, \emph{Property {$C''$} and {F}unction {S}paces}, Proceedings of
  the American Mathematical Society \textbf{104} (1988), no.~3, 917--919.

\bibitem{Scheepers2003}
Marion Scheepers, \emph{Selection principles and covering properties in topology},
  Note di Matematica \textbf{22} (2003), no.~2.

\bibitem{Scheepers}
\bysame, \emph{Combinatorics of open covers ({III}): games,
  {C}p({X})}, Fundamenta Mathematicae \textbf{152} (1997), no.~3, 231--254.

\end{thebibliography}
\end{document}